\pdfoutput=1                         
\documentclass[11pt]{article}

\usepackage[margin=1in]{geometry}     
\usepackage[utf8]{inputenc}
\usepackage[T1]{fontenc}

\usepackage{amsmath}
\usepackage{amssymb}
\usepackage{amsfonts}
\usepackage{amsthm}
\usepackage{mathtools}
\usepackage{bm}

\usepackage{graphicx}
\usepackage{xcolor}
\usepackage{multirow}
\usepackage{booktabs}
\usepackage{float}
\usepackage{subcaption}

\usepackage[ruled,vlined,linesnumbered]{algorithm2e}
\usepackage{pgf}
\usepackage{pgfplots}
\pgfplotsset{compat=1.18}
\usepackage{tikz}

\usepackage{listings}
\usepackage{upquote}                 
\newfloat{listing}{tbp}{lol}          
\floatname{listing}{Figure}
\definecolor{codekw}{HTML}{008000}
\definecolor{codestr}{HTML}{BA2121}
\definecolor{codecom}{HTML}{408080}
\definecolor{codeframe}{HTML}{BFBFBF}
\usepackage{titlesec}
\titlelabel{\thetitle. }

\usepackage{hyperref}
\hypersetup{
    colorlinks,
    linkcolor={red!50!black},
    citecolor={blue!50!black},
    urlcolor={blue!80!black}
}

\usepackage{natbib}
\bibpunct[, ]{(}{)}{,}{a}{}{,}%
\DeclareMathOperator*{\argmin}{arg\,min}

\usetikzlibrary{decorations.pathreplacing}
\tikzset{
  show curve controls/.style={
    decoration={
      show path construction,
      curveto code={
        \fill[red] (\tikzinputsegmentsupporta) circle(1.5pt);
        \fill[red] (\tikzinputsegmentsupportb) circle(1.5pt);
        \draw[black]
        (\tikzinputsegmentfirst)
        .. controls (\tikzinputsegmentsupporta)
                and (\tikzinputsegmentsupportb) ..
        (\tikzinputsegmentlast);
      }
    }, decorate
  }
}
\usetikzlibrary{positioning,shapes.multipart,arrows.meta,fit,backgrounds,calc}

\graphicspath{{figures/}{./}}

\theoremstyle{plain}
\newtheorem{theorem}{Theorem}
\newtheorem{proposition}{Proposition}
\newtheorem{corollary}{Corollary}
\newtheorem{lemma}{Lemma}
\theoremstyle{definition}
\newtheorem{definition}{Definition}
\newtheorem{example}{Example}
\theoremstyle{remark}

\renewenvironment{proof}[1]{\par\noindent{\itshape #1}~}{\par\medskip}
\newcommand{\Halmos}{\hfill\ensuremath{\square}}
\newcommand{\Equationvalidatefalse}{}

\newenvironment{APPENDIX}[1]{\par\bigskip\appendix}{\par}

\begin{document}

\title{Computational Framework for B\'ezier Distributions}

\author{%
  Esteban Leiva\thanks{Department of Industrial and Systems Engineering,
    University of Southern California, Los Angeles, CA, USA,
    \texttt{leivamon@usc.edu}}
  \and
  Andr\'es L. Medaglia\thanks{Centro para la Optimizaci\'on y Probabilidad
    Aplicada (COPA), Departamento de Ingenier\'ia Industrial, Universidad de los
    Andes, Bogot\'a, Colombia, \texttt{amedagli@uniandes.edu.co}}
  \and
  Luis F. Zuluaga\thanks{Department of Industrial and Systems Engineering,
    Lehigh University, Bethlehem, PA, USA, \texttt{luis.zuluaga@lehigh.edu}}%
}

\date{\today}
\maketitle
\begin{abstract}
    Flexible continuous univariate distributions with bounded support are essential for accurate input modeling in stochastic simulation and decision analysis. Although B\'ezier distributions provide a powerful family capable of representing complex shapes, their adoption has been hindered by the lack of efficient fitting procedures and modern software implementations. This paper develops a computational framework for fitting B\'ezier distributions to empirical data via both minimum error and maximum likelihood estimation, leveraging first-order optimization methods and exploiting the geometry of the parameter space. We identify provably (asymptotically) lossless convex restrictions of the feasible set that enable efficient projection operators based on isotonic regression and develop first-order algorithms that reduce computational runtime by three to four orders of magnitude compared to traditional derivative-free methods, while delivering consistent fits across real-world data. When benchmarked against the nonlinear solver IPOPT, our methods prove three orders of magnitude faster on average and more robust, while achieving comparable accuracy. To bridge the gap between theory and practice, we introduce \texttt{bezierv}, an open-source Python package providing a unified interface for fitting, analyzing, and convolving B\'ezier distributions.
\end{abstract}


\section{Introduction}

Flexible continuous univariate distributions with bounded support are pivotal across decision sciences, economics, quality control, Bayesian inference, and biostatistics \citep{Gupta2004}. They are particularly relevant for input modeling in stochastic simulation, where standard parametric families often suffer from three limitations: i) difficulty to fit a real-world input process (e.g., multimodality, troublesome tails); ii) difficulty to estimate the distribution parameters from data; iii) limited expressive power to match the input process's shape \citep{Kuhl2010, Biller2010}. Among the most widely used flexible distributions in the simulation literature are the generalized beta distribution \citep{AbouRizk1994, Johnson1995, Gupta2004} and the Johnson translation system \citep{Johnson1949, DeBrota1989, Chen2001}. Motivated by these shortcomings, \citet{Wagner1993} parameterized the cumulative distribution function (cdf) as a B\'ezier curve, introducing the \emph{B\'ezier distribution}: a flexible family that, though less widely adopted, overcomes the three shortcomings \citep{Kuhl2010}.

B\'ezier distributions are built upon B\'ezier curves: smooth parametric curves determined by control points. These curves derive from the Bernstein polynomials \citep{Bernstein1912} and were developed independently by P.\ de Casteljau and P.\ Bézier at Citroën and Renault \citep{Farin1997}, and are now ubiquitous in computer-aided geometric design for their ease of construction and expressiveness. \citet{Wagner1993} parameterized the random variable's cdf through the curve's exact interpolation and closed-form derivative. They also developed \texttt{PRIME}, a Windows 3.0 application for data-driven fitting and sampling of B\'ezier distributions. These distributions were later extended to bivariate populations \citep{Wagner1995, Wagner1996WSC}, benchmarked against classical flexible families \citep{Kuhl2007, Kuhl2010}, and applied in soft computing \citep{Medaglia2002}, yet never became mainstream: no new fitting methods or modern software have appeared since \texttt{PRIME}, and this absence of computational tools has hindered their adoption \citep{Law_2015}.

In this paper we revisit B\'ezier distributions through the lens of mathematical optimization. We introduce a convex restriction of the estimation problem that constrains the $n+1$ control points, defining the $n$-th degree B\'ezier curve, to the monotone set of \citet{Wagner1993} (see \eqref{eq:C_n}), and prove it asymptotically lossless: although monotone control points form a strict subset of the set of B\'ezier distributions at any fixed degree, every feasible B\'ezier distribution is approximated uniformly by one with monotone control points, with error decaying at rate $\mathcal{O}(1/n)$ in the degree. We leverage this to show that high-degree B\'ezier distributions with monotone control points approximate the empirical fit of general ones arbitrarily well. Within this framework we develop first-order methods for minimum error and maximum likelihood estimation that reduce runtime by three to four orders of magnitude relative to the original derivative-free methods  of \cite{Wagner1993}, while improving fit quality. The restricted formulations are also amenable to off-the-shelf nonlinear solvers such as IPOPT \citep{Wachter2006}, which we use as a benchmark to show that our methods are substantially faster and more robust, while IPOPT attains better accuracy at the cost of significantly longer runtimes and solver failures that become frequent on real-world data. Beyond estimation, we motivate a Monte Carlo scheme for aggregating independent B\'ezier variables, a capability that proves essential when embedding them in transportation problems, such as the stochastic shortest path. Finally, we develop \texttt{bezierv} (\url{https://github.com/EstebanLeiva/bezierv}), an open-source Python package with a unified interface for fitting, analyzing, and convolving B\'ezier distributions, and demonstrate it on multimodal input modeling, synthetic data from a wide range of families, real transportation data, and the data-driven reliable shortest path problem.

The paper is structured as follows. Section~\ref{sec:cdf} presents the B\'ezier distribution. Section~\ref{sec:fitting} develops our minimum error estimation methods by imposing monotonicity on the control points to restrict the problem to a tractable optimization problem and proving this restriction asymptotically lossless. Section~\ref{sec:mle} extends the methodology to maximum likelihood estimation. Section~\ref{sec:conv} discusses the convolution of B\'ezier random variables and motivates the Monte Carlo aggregation used in Section~\ref{sec:spp}. Section~\ref{sec:package} introduces the \texttt{bezierv} Python package. Section~\ref{sec:experiments} reports computational experiments on real and synthetic datasets that compare our algorithms with those of \citet{Wagner1993} and IPOPT, compare the B\'ezier fit against other widely used flexible distributions, and confirm the approximation guarantee for the monotonic restriction. Section~\ref{sec:spp} applies the package to data-driven reliable shortest path problems. We finish in Section~\ref{sec:conclusions} with concluding remarks.

\section{B\'ezier distribution} \label{sec:cdf}
    A B\'ezier curve of degree $n$ is defined parametrically by
    \begin{equation}
        \mathbf{P}(t) = [
                        P_x(t;n, \mathbf{x}),
                        P_z(t;n, \mathbf{z})
                        ]^\top
                        = \sum\limits_{i=0}^n B_{n,i}(t) \mathbf{p}_i, \quad t\in [0,1], \label{eq:beziercurve}
    \end{equation}
    where $\mathcal{P} = \{\mathbf{p}_0, ..., \mathbf{p}_n\}$ is the set of $n+1$ control points, $\mathbf{p}_i = (x_i, z_i)^\top$ for $i=0,1,...,n$, $\mathbf{x} = (x_0,x_1,...,x_n)^\top$, $\mathbf{z} = (z_0,z_1,...,z_n) ^\top$, and the blending function $B_{n,i}(t)$ is the basis Bernstein polynomial
    \[
    B_{n,i}(t)=
    \begin{cases}
    \dfrac{n!}{i!\,(n-i)!}\,t^{\,i}\bigl(1-t\bigr)^{n-i}, & \text{for all } t\in[0,1]\ \text{and}\ i=0,1,\ldots,n,\\[6pt]
    0, & \text{otherwise}.
    \end{cases}
    \]
    B\'ezier curves are continuous, as they are a weighted sum of polynomials. For $t=0$,  we have that $B_{n,0}(0) =1$ and $B_{n,i}(0) =0$ for $i=1,\ldots,n$; whereas for $t=1$, $B_{n,i}(1) =0$ for $i=0,\ldots,n-1$ and $B_{n,n}(1) =1$. Hence, B\'ezier curves satisfy the endpoint-interpolation property, namely, $\mathbf{P}(0) = \mathbf{p}_0$ and $\mathbf{P}(1) = \mathbf{p}_n$.  Other notable properties of B\'ezier curves include: (i) affine invariance, (ii) invariance under affine parameter transformations, (iii) convex hull property, (iv) symmetry, (v) invariance under barycentric combinations, (vi) linear precision, and (vii) pseudo-local control. For a detailed explanation of the properties above and an in-depth study of B\'ezier curves the reader is referred to \citet{Farin1997}.

    \subsection{B\'ezier distribution and density functions}
    Let $X$ be a univariate continuous random variable supported on the bounded interval $[a,b]$. We say $X$ is a B\'ezier random variable if its cumulative distribution function (cdf) is given parametrically by a B\'ezier curve \eqref{eq:cdf}
    \begin{equation}
    x(t) = \sum_{i=0}^{n} B_{n,i}(t)\,x_i, \qquad F_{X}(x(t)) = \sum_{i=0}^{n} B_{n,i}(t)\,z_i, \qquad t \in [0,1]. \label{eq:cdf}
    \end{equation}

    Recall that a cdf must satisfy the following properties: (i) it is non-decreasing, (ii) it is right-continuous, and (iii) it satisfies the boundary conditions \( F_X(a) = 0 \) and \( F_X(b) = 1 \). Since B\'ezier curves are continuous, it suffices to enforce the monotonicity and endpoint conditions to ensure the validity of the B\'ezier cdf. The latter can be satisfied by setting the control points \( \mathbf{p}_0 = [a, 0]^\top \) and \( \mathbf{p}_n = [b, 1]^\top \) due to the endpoint-interpolation property.

    To enforce monotonicity we have to ensure that the derivative of the B\'ezier curve representing the cdf \eqref{eq:cdf} be nonnegative throughout the domain. This derivative corresponds to the random variable's probability density function (pdf) and is defined parametrically by
    \begin{equation}
    x(t) = \sum_{i=0}^{n} B_{n,i}(t)\,x_i, \qquad f_{X}(x(t)) = \frac{P_z'(t;n,\mathbf{z})}{P_x'(t;n,\mathbf{x})} = \frac{\sum_{i=0}^{n-1} B_{n-1,i}(t) \Delta z_i}{\sum_{i=0}^{n-1} B_{n-1,i}(t)  \Delta x_i}, \quad t \in [0,1],\label{eq:pdf}
    \end{equation}
    where 
    \begin{equation*}
    \begin{aligned}
     \Delta z_i = z_{i+1}-z_{i} \quad \text{ and } \quad \Delta x_i = x_{i+1}-x_{i},
    \end{aligned}
    \end{equation*}
    are the first differences of the $x$- and $z$-coordinates of the control points $\mathbf{p}_{i+1}$ and $\mathbf{p}_{i}$ for $i=0,\ldots,n-1$.
    Note that while $ f_X(x(t))$ depends both on $\mathbf{x}$ and $\mathbf{z}$, $x(t)$ and $F_X(x(t))$ depend solely on $\mathbf{x}$ and $\mathbf{z}$, respectively. To highlight this dependency we will write $x(t) = P_x(t;n,\mathbf{x})$, $F_X(x(t)) = P_z(t;n,\mathbf{z})$, and $f_X(x(t)) = f_X(t;n,\mathbf{x},\mathbf{z})$. For illustrative examples of B\'ezier distributions, see Figure \ref{fig:multimodal_fit} and Appendix \ref{appendix:examples_bezier}.
    
\section{Minimum error estimation} \label{sec:fitting}
Let us consider a sample $\{ X_j\}_{j=1}^m$ of a continuous random variable $X$ with unknown cdf $F_X(\cdot)$. Assume that the sample contains at least two distinct values. The sorted sample (ascending order) is denoted by  $\{ X_{(j)}\}_{j=1}^m$, where $X_{(1)}$ and $X_{(m)}$ are the smallest and largest values from the sample, respectively. The empirical cdf is given by $F_m(x) = \frac{1}{m} \sum_{j=1}^m \mathbf{1}(X_{(j)} \leq x)$, where $\mathbf{1}(X_{(j)} \leq x) = 1$  if $X_{(j)} \leq x$ and $0$, otherwise. 

Given a fixed number of control points $n$ with  $n \ll m$, we aim to fit a B\'ezier distribution as defined in \eqref{eq:cdf} to the empirical distribution function $F_m(\cdot)$. The objective of minimum error estimation is to determine the location of the control points $\mathbf{p}_i = \left(x_i,z_i\right)^\top$ for $i=0,\ldots,n$ that minimizes the mean squared error (MSE) between the fitted and empirical cdfs at the sample points. This problem is formulated as the semi-infinite nonlinear optimization problem \eqref{eq:nonlinear}

 {%
  \setlength{\abovedisplayskip}{0pt}%
  \setlength{\belowdisplayskip}{0pt}%
  \setlength{\abovedisplayshortskip}{0pt}%
  \setlength{\belowdisplayshortskip}{0pt}%
  \setlength{\jot}{1pt}
  \begin{subequations} \label{eq:nonlinear}
  \begin{align}
      \min\limits_{\mathbf{x}, \mathbf{z}, \mathbf{t}} \quad & \frac{1}{m}\sum_{j=1}^{m} \left( F_m(X_{(j)})  - P_z(t_{(j)}; n, \mathbf{z})\right)^2
  \label{obj:sumofsquares} \\
      \text{s.t.} \quad & P_x( t_{(j)}; n, \mathbf{x}) = X_{(j)}, \quad 0\leq t_{(j)} \leq 1, & & j=1,\ldots,m \label{constraint:sample_ts}\\
      & f_X(t;n,\mathbf{x}, \mathbf{z}) \geq 0, & & \forall t \in [0, 1] \label{constraint:positivepdf} \\
      & z_0 = 0, \   z_n = 1 , \quad  x_0 = X_{(1)}, \ x_n = X_{(m)}. & & \label{constraint:eqx_0m}
  \end{align}
  \end{subequations}
  }%

Note that by the parametric nature of the B\'ezier cdf we use \( t_{(j)} \) instead of \( t \) to emphasize that the parameter is computed at each sample point by finding \( t_{(j)} \in [0,1] \) such that \( x(t_{(j)}) = P_x( t_{(j)}; n, \mathbf{x}) = X_{(j)} \) for \( j = 1, \dots, m \) in~\eqref{constraint:sample_ts}. Constraint \eqref{constraint:positivepdf} ensures that the pdf is non-negative and constraint~\eqref{constraint:eqx_0m} ensures the boundary conditions. Since the pdf is the ratio \( f_X(t;n,\mathbf{x},\mathbf{z}) = P_z'(t;n,\mathbf{z})/P_x'(t;n,\mathbf{x}) \), constraint \eqref{constraint:positivepdf} technically amounts to \( P_x'(t;n,\mathbf{x}) > 0 \) for \( t \in (0,1) \) together with \( P_z'(t;n,\mathbf{z}) \geq 0 \) for \( t \in [0,1] \).

\subsection{The restricted problem} \label{sec:subset}
Let $\mathcal{S}_n$ denote the set of $n+1$ control points whose $n$-th degree B\'ezier curve is a valid cdf on the support $[X_{(1)}, X_{(m)}]$
{\Equationvalidatefalse
\begin{equation*}
\begin{aligned}
\mathcal{S}_n = \Big\{(\mathbf{x}, \mathbf{z}) \in \mathbb{R}^{(n+1)}\times \mathbb{R}^{(n+1)} \mid \ &x_0 = X_{(1)}, \: x_n = X_{(m)}, \: z_0 = 0, \: z_n=1, \\
&P_x'(t;n,\mathbf{x}) > 0 \ \ \forall t \in (0,1), \; P_z'(t;n,\mathbf{z})\geq 0 \ \ \forall t \in [0,1] \Big\}.
\end{aligned}
\end{equation*}}
The two semi-infinite constraints in $\mathcal{S}_n$ pose a significant computational challenge. We therefore replace them with the \emph{monotonicity} condition of \citet{Wagner1993}, in which the control points satisfy the boundary conditions $x_0 = X_{(1)}$, $x_n = X_{(m)}$, $z_0 = 0$, $z_n = 1$, and are ordered such that $x_i \leq x_{i+1}, z_i \leq z_{i+1}$ for $i=0,...,n-1$. Note that this ordering implies that $\Delta x_i\geq 0$ and $\Delta z_i\geq 0$ for $i=0,...,n-1$.

The monotonicity condition implies membership in $\mathcal{S}_n$, reduces to finitely many linear constraints on the control points, and the B\'ezier curves it admits approximate those generated by $\mathcal{S}_n$ arbitrarily well as the degree increases (see Theorem \ref{theorem}). To establish these properties, we define the set of $n+1$ ordered control points as
{\Equationvalidatefalse
\begin{equation} \label{eq:C_n}
\begin{aligned}
    \mathcal{C}_n = \Big\{(\mathbf{x},\mathbf{z})\in \mathbb{R}^{(n+1)}\times \mathbb{R}^{(n+1)} \mid \ &x_0 = X_{(1)},\ x_n = X_{(m)},\ z_0 = 0,\ z_n=1, \\
    &x_i \leq x_{i+1},\ z_i \leq z_{i+1}, \ i=0,\ldots,n-1 \Big\}.
\end{aligned}
\end{equation}}
\begin{lemma}\label{lem:CsubsetS}
$\mathcal{C}_n \subseteq \mathcal{S}_n$.
\end{lemma}
\begin{proof}{Proof.}
    See Appendix \ref{appendix:proofs}.
\end{proof}

Enforcing monotonicity of the control points is a sufficient condition for $(\mathbf{x}, \mathbf{z}) \in \mathcal{S}_n$ and unlike $\mathcal{S}_n$, $\mathcal{C}_n$ is defined by finitely many linear constraints. Example \ref{counterexample} shows that the monotonicity of the control points is not a necessary condition for membership in $\mathcal{S}_n$.
\begin{example} \label{counterexample}
    Let $n=3$ and define the control points $\mathcal{P}_3=\{(0,0), (0.6,1),(0.4,0), (1,1)\}$. The first differences are $\Delta x_0= 0.6$, $\Delta x_1=-0.2$, $\Delta x_2= 0.6$, $\Delta z_0= 1$, $\Delta z_1=-1$, and $\Delta z_2=1$. Taking derivatives, $P_z'(t;3,\mathbf{z})=3(2t-1)^2\geq 0$ on $[0,1]$ and $P_x'(t;3,\mathbf{x})=3(1.6t^2-1.6t+0.6)>0$ on $(0,1)$, so $\mathcal{P}_3\in\mathcal{S}_3\setminus\mathcal{C}_3$. Moreover, since the Bernstein polynomials $\{B_{3,i}\}^3_{i=0}$ form a basis for the space of polynomials of degree at most $n=3$ on $[0,1]$, the B\'ezier curve defined by $\mathcal{P}_3$ is unique: no other set of four control points defines the same third-degree curve.
\end{example}

Although $\mathcal{C}_{\tilde{n}} \subsetneq \mathcal{S}_{\tilde{n}}$ is strict for any fixed degree $\tilde{n}$, Theorem~\ref{theorem} shows that the gap closes as the approximating degree $n$ grows: every B\'ezier curve from $\mathcal{S}_{\tilde{n}}$ is approximated uniformly, at rate $\mathcal{O}(1/n)$, by one with control points in $\mathcal{C}_n$. Together with Lemma~\ref{lem:CsubsetS}, this makes $\mathcal{C}_n$ a finite convex inner approximation of $\mathcal{S}_n$ that is \emph{asymptotically tight}.

\begin{theorem} \label{theorem}
  Let $(\tilde{\mathbf{x}}, \tilde{\mathbf{z}}) \in \mathcal{S}_{\tilde{n}}$ define a B\'ezier cdf of degree $\tilde{n}$ denoted by $\tilde{\mathbf{P}}(t) = \bigl[P_x(t;\, \tilde{n},\, \tilde{\mathbf{x}}),\, P_z(t;\, \tilde{n},\, \tilde{\mathbf{z}})\bigr]^\top$ and let $\kappa = \max\!\bigl( \| P_x''(\cdot;\, \tilde{n},\, \tilde{\mathbf{x}}) \|_\infty,\; \| P_z''(\cdot;\, \tilde{n},\, \tilde{\mathbf{z}}) \|_\infty \bigr)$. For any $n \geq 1$, let
  $(\mathbf{x}, \mathbf{z}) \in \mathbb{R}^{(n+1)} \times \mathbb{R}^{ (n+1)}$ be the control points obtained by evaluating $\tilde{P}$ at $i/n$,
  \begin{equation}
  x_i \;=\; P_x\!\left(\tfrac{i}{n};\,\tilde{n},\,\tilde{\mathbf{x}}\right),
  \qquad
  z_i \;=\; P_z\!\left(\tfrac{i}{n};\,\tilde{n},\,\tilde{\mathbf{z}}\right),
  \qquad i = 0, 1, \ldots, n,
  \label{eq:bernstein-cp}
  \end{equation}
  and let $\mathbf{P}(t) := \bigl[P_x(t;\, n,\, {\mathbf{x}}),\, P_z(t;\, n,\, {\mathbf{z}})\bigr]^\top$
  be the $n$-th degree B\'ezier curve they generate. Then $({\mathbf{x}}, {\mathbf{z}}) \in \mathcal{C}_n$ and
  \begin{equation}
  \sup_{t \in [0,1]} \bigl\| \mathbf{P}(t) - \tilde{\mathbf{P}}(t) \bigr\|_\infty
  \;\leq\;
  \frac{\kappa}{8n}.
  \label{eq:rate}
  \end{equation}
\end{theorem}
\begin{proof}{Proof.}
    See Appendix \ref{appendix:proofs}.
\end{proof}

Corollary~\ref{corollary} shows that the MSE gap between the fits in $\mathcal{C}_n$ and $\mathcal{S}_{\tilde n}$ also vanishes at rate $\mathcal{O}(1/n)$, so high-degree approximations in $\mathcal{C}_n$ can reproduce the empirical fit of elements of $\mathcal{S}_{\tilde n}$ arbitrarily well. Experimentally, the restriction barely affects fit quality and a few control points already yield very low MSE (see Sections~\ref{sec:fitquality} and~\ref{sec:experiments}).

\begin{corollary} \label{corollary}
    Let $(\tilde{\mathbf{x}}, \tilde{\mathbf{z}}) \in \mathcal{S}_{\tilde{n}}$ define a B\'ezier cdf with bounded pdf, i.e.
    $\tilde{\delta} := \min_{t\in[0,1]} P_x'(t;\,\tilde{n},\,\tilde{\mathbf{x}}) > 0$,
    and let $(\mathbf{x}, \mathbf{z}) \in \mathcal{C}_n$ be the control points~\eqref{eq:bernstein-cp}
    of Theorem~\ref{theorem}. Set
    \sloppy$\rho = \tfrac{\Vert P_z'(\cdot ; \tilde n, \tilde{\mathbf{z}})\Vert_\infty}{\tilde{\delta}}$
    and $\kappa = \max\bigl( \|P_x''(\cdot;\, \tilde{n},\, \tilde{\mathbf{x}})\|_\infty,\, \|P_z''(\cdot;\, \tilde{n},\, \tilde{\mathbf{z}})\|_\infty \bigr)$.
    Then
    \begin{equation}
        \Bigg| \frac{1}{m}\sum_{j=1}^m \bigl( F_m(X_{(j)}) - P_z(t_{(j)};\,n,\,\mathbf{z}) \bigr)^2
        - \frac{1}{m}\sum_{j=1}^m \bigl( F_m(X_{(j)}) - P_z(\tilde{t}_{(j)};\,\tilde{n},\,\tilde{\mathbf{z}}) \bigr)^2 \Bigg|
        \leq \frac{(1+ \rho)\kappa}{4n}, \label{eq:bound_cor}
    \end{equation}
    where $t_{(j)}$ and $\tilde{t}_{(j)}$ satisfy
    $P_x(t_{(j)};\,n,\,\mathbf{x})= X_{(j)} = P_x(\tilde{t}_{(j)};\,\tilde{n},\,\tilde{\mathbf{x}})$.
\end{corollary}
\begin{proof}{Proof.}
    See Appendix \ref{appendix:proofs}.
\end{proof}

We therefore restrict the search to $\mathcal{C}_n$ by replacing the two semi-infinite constraints with $2n$ linear inequalities to obtain the restricted fitting problem \eqref{eq:restricted}:

 {%
  \setlength{\abovedisplayskip}{3pt}%
  \setlength{\belowdisplayskip}{3pt}%
  \setlength{\abovedisplayshortskip}{0pt}%
  \setlength{\belowdisplayshortskip}{0pt}%
  \setlength{\jot}{1pt}
  \begin{subequations} \label{eq:restricted}
  \begin{align}
      \min_{\mathbf{x},\mathbf{z}, \mathbf{t}} \quad & \frac{1}{m} \sum_{j=1}^{m}\left( F_m(X_{(j)}) - P_z(t_{(j)};n,\mathbf{z}) \right)^2
  \label{obj:nonlinear} \\
      \text{s.t.} \quad
      & P_x( t_{(j)};n,\mathbf{x}) = X_{(j)}, \quad 0 \leq t_{(j)} \leq 1, && j=1,\ldots,m \label{constraint:sample} \\
      & x_i \leq x_{i+1}, \quad z_i \leq z_{i+1}, && i=0,\ldots,n-1 \label{constraint:orderingx} \\
      & z_0 = 0, \; z_n = 1, \quad x_0 = X_{(1)}, \; x_n = X_{(m)}. \label{constraint:z_n=1}
  \end{align}
  \end{subequations}
  }%

\noindent Although the parametric equality constraint \eqref{constraint:sample} renders the optimization problem nonconvex, the restricted formulation remains tractable. It can be solved using standard off-the-shelf nonlinear solvers, such as IPOPT \citep{Wachter2006}.

\subsection[Projecting onto C]{Projecting onto $\mathcal{C}$} \label{sec:projection}
To leverage the convexity of set $\mathcal{C}_n$, we propose a projected gradient descent (PGD) methodology. The viability of this approach relies on the existence of an efficient projection onto $\mathcal{C}_n$. Unless otherwise stated, the degree is assumed to be fixed, and we write $\mathcal{C}$ in place of $\mathcal{C}_n$ ($\mathcal{S}$ in place of $\mathcal{S}_n$) for notational simplicity.

The set $\mathcal{C}$ factors as $\mathcal{C} = \mathcal{C}_x \times \mathcal{C}_z$, where
{\Equationvalidatefalse
\begin{equation*}
\begin{aligned}
    \mathcal{C}_x &= \{\mathbf{x}\in\mathbb{R}^{(n+1)} \mid x_0 = X_{(1)}, x_n = X_{(m)}, x_i \leq x_{i+1}, i=0,\ldots,n-1\},\\
    \mathcal{C}_z &= \{\mathbf{z}\in\mathbb{R}^{(n+1)} \mid z_0 = 0, z_n = 1, z_i \leq z_{i+1}, i=0,\ldots,n-1\}.
\end{aligned}
\end{equation*}}
Since $\mathcal{C}_x$ and $\mathcal{C}_z$ do not have coupling constraints, the projection onto $\mathcal{C}$ decomposes into the cartesian product of the componentwise projections $\pi_{\mathcal{C}_x}(\mathbf{a})$ and $\pi_{\mathcal{C}_z}(\mathbf{b})$. Without loss of generality, we focus on computing $\pi_{\mathcal{C}_x}(\mathbf{a})$, for which we introduce the well-known isotonic regression problem \citep{Brunk1955, Barlow1972} defined by the quadratic program \eqref{eq:iso}
{\Equationvalidatefalse
\begin{equation*}
\min_{\mathbf{x}\in\mathbb{R}^{n+1}} \Bigl\{ \textstyle\sum_{i=0}^{n}\left( a_i - x_i \right)^2 \ : \ x_i \leq x_{i+1},\ i=0,\ldots,n-1; \ \ L \leq x_i \leq U,\ i=0,\ldots,n \Bigr\}. \tag{ISO}\label{eq:iso}
\end{equation*}}
Isotonic regression is efficiently solvable in $\mathcal{O}(n)$ time using the pool adjacent violators algorithm (PAVA) \citep{Barlow1972}, for which standard implementations are available \citep{Balabdaoui2011}.

Standard PAVA implementations do not natively accommodate the boundary equality constraints $x_0=L$ and $x_n=U$. However, as we notice in Proposition \ref{proposition:projeq}, we can reformulate this constrained problem into a standard isotonic regression in $(n-1)$ variables. This observation enables the direct application of PAVA to efficiently compute the projection $\pi_{\mathcal{C}_x}(\mathbf{a})$.
\begin{proposition} \label{proposition:projeq}
    Solving problem \eqref{eq:iso} with the equality constraints $x_0 = L$ and $x_n = U$ is equivalent to solving the $(n-1)$-variable isotonic regression problem
    \begin{equation} \label{eq:iso-equiv}
\min_{\mathbf{x} \in \mathbb{R}^{n-1}} \Bigl\{ \textstyle\sum_{i=1}^{n-1}\left( a_i - x_i \right)^2 \ : \ x_i \leq x_{i+1},\ i=1,\ldots,n-2; \ \ L \leq x_i \leq U,\ i=1,\ldots,n-1 \Bigr\}.
    \end{equation}
 In particular, if $\mathbf{x}^*$ is a solution for \eqref{eq:iso-equiv}, then $(L, \mathbf{x}^*, U)$ is a solution for \eqref{eq:iso} with equality constraints $x_0 = L$ and $x_n = U$.
\end{proposition}

\subsection{Projected gradient descent} \label{sec:PGD}
While general-purpose solvers offer a practical approach for solving \eqref{eq:restricted}, the method is not reliable on real-world data and is not fast enough for real-time and large-scale applications, as demonstrated in Section~\ref{sec:experiments}. To address this, we introduce a formulation that ensures convexity by fixing $\mathbf{x}$ and allows the use of a PGD algorithm with convergence guarantees. As mentioned before, the constraints in~\eqref{constraint:sample} are required to evaluate the objective function \eqref{obj:nonlinear} at the sample points \( X_{(j)} \). However, by fixing the $x$-coordinates of the control points, the corresponding parameters $t_{(j)}$ become constant for all $j=1,...,m$. This restriction reduces the original nonlinear problem to the convex quadratic problem below
\begin{equation} \label{eq:convexQP}
\min_{\mathbf{z}} \Bigl\{ \textstyle\frac{1}{m} \sum_{j=1}^{m}\left( F_m(X_{(j)}) - P_z(t_{(j)};n,\mathbf{z}) \right)^2 \ : \ z_i \leq z_{i+1},\ i=0,\ldots,n-1; \ \ z_0 = 0,\ z_n = 1 \Bigr\}.
\end{equation}
Note that \eqref{eq:convexQP} further restricts the feasible set of \eqref{eq:restricted}, so its optimal MSE is necessarily no smaller; additionally, the approach requires a principled choice for fixing the $x$-coordinates. Fixing $\mathbf{x}$ also decouples the problem size from the sample size: the $m$ parameters $t_{(j)}$ and their associated constraints in~\eqref{constraint:sample} are eliminated, leaving only the $n+1$ variables $\mathbf{z}$ regardless of $m$, which is what renders the formulation tractable as the sample grows.

Since $t_{(j)}$ is fixed for all $j=1,...,m$, we define $\mathbf{b}_j \in \mathbb{R}^{n+1}$ with entries $(\mathbf{b}_j)_i = B_{n,i}(t_{(j)})$ for $i=0,\dots,n$, so that the expression $P_z(t_{(j)};n,\mathbf{z}) = \sum_{i=0}^{n} B_{n,i}(t_{(j)})\,z_i =\mathbf{b}^\top_j \mathbf{z}$ is clearly linear in $\mathbf{z}$. Consequently, problem \eqref{eq:convexQP} reduces to a classic constrained linear least squares problem

\[\min_\mathbf{z}\left\{\frac{1}{m}\Vert \mathbf{c} - \mathbf{B}\mathbf{z} \Vert_2^2 \ \text{ s.t. } z_0 = 0, \ z_n = 1, \ z_i \leq z_{i+1}, \ i=0,\ldots,n-1\right\},\]
where $c_j = F_m(X_{(j)})$ for each $j=1,...,m$ and $\mathbf{B}\in \mathbb{R}^{m\times (n+1)}$ is a matrix with rows $\mathbf{b}_j$ for $j=1,...,m$. The gradient of $f(\mathbf{z}) = \frac{1}{m}\Vert \mathbf{c} - \mathbf{B}\mathbf{z} \Vert_2^2$ is given by $\nabla f(\mathbf{z}) = \frac{2}{m}\mathbf{B}^\top(\mathbf{B}\mathbf{z}-\mathbf{c})$ and the Hessian is $\nabla^2f(\mathbf{z}) = \frac{2}{m}\mathbf{B}^\top\mathbf{B}$. Furthermore, $f(\mathbf{z})$ is $\beta$-smooth, that is $\nabla f(\mathbf{z})$ is $\beta$-Lipschitz, with $\beta= \frac{2}{m}\lambda_{\text{max}}(\mathbf{B}^\top\mathbf{B})$.

Given the computational efficiency of projecting onto the constraint set $\mathcal{C}_z$, we employ a projected gradient descent algorithm to solve problem \eqref{eq:convexQP}. As mentioned before, prior to execution, the variables \(\mathbf{x}\) must be fixed using a principled, low-cost heuristic. For this purpose, we choose \(x_i\) as the empirical \(\tfrac{i}{n}\)-quantile of the sample $\{X_{j}\}_{j=1}^m$. Given the fixed $x$-coordinates and the corresponding sample parameters $t_{(j)}$ for $j=1,...,m$, Algorithm \ref{alg:projgraditer} receives as input the initial value of the $z$-coordinates of the control points $\mathbf{z}^0$. In each iteration, the algorithm performs a gradient descent update followed by the projection $\Pi_{\mathcal{C}_z}$ (line 3) that enforces the ordering, end-point, and monotonicity constraints using PAVA and the methodology outlined in Section \ref{sec:projection}.

\begin{algorithm}[htbp]
    \SetKwInput{KwParameters}{Parameters}
    \SetKwData{KwRequire}{Require}
    \SetKwData{KwEnsure}{Ensure}

    \caption{PGD}\label{alg:projgraditer}
    \KwParameters{step-size $\eta>0$; number of iterations $N$.}
    \KwIn{initial $z$-coordinates $\mathbf{z}^0$.}

    \For{$k = 0$ to $N-1$}{
        Compute $\nabla f(\mathbf{z}^k)$

        $\mathbf{z}^{k+1} \leftarrow \Pi_{\mathcal{C}_z}\left(\mathbf{z}^{k} - \eta\nabla f(\mathbf{z}^k) \right)$
    }
    \Return $\mathbf{z}^N$
\end{algorithm}

It is well-established that projected gradient descent applied to a convex, $\beta$-smooth function over a convex set with constant step-size $\eta = 1/\beta$ has a sublinear convergence rate in the number of iterations \citep{Nesterov2004, bubeck2015}. Therefore, to achieve a desired accuracy level $\epsilon > 0$, the number of iterations required scales as $N = \mathcal{O}(1/\epsilon)$.

\section{Maximum likelihood estimation} \label{sec:mle}
Maximum likelihood estimation (MLE) is the standard approach for fitting parametric distributions in simulation input modeling \citep{Law_2015}. We now adapt the restriction framework developed for minimum error estimation to this setting. Given the sample $\{X_{(j)}\}_{j=1}^m$, we aim to maximize the joint probability density of the observed sample points. The MLE fitting problem can be formulated as
\begin{equation}
\min_{\mathbf{x},\mathbf{z},\mathbf{t}}\left\{ \mathcal{L}(\mathbf{x},\mathbf{z}) = -\sum_{j=1}^{m} \ln f_{X}(t_{(j)}; n, \mathbf{x}, \mathbf{z}) \quad \text{s.t.} \quad \eqref{constraint:sample_ts}-\eqref{constraint:eqx_0m} \right\},
\end{equation}
where \eqref{constraint:sample_ts}-\eqref{constraint:eqx_0m} define the same feasible set of problem \eqref{eq:nonlinear}.
Substituting the B\'ezier pdf \eqref{eq:pdf}, the negative of the log-likelihood function takes the form
\begin{equation}
\mathcal{L}(\mathbf{x},\mathbf{z}) = \sum_{j=1}^{m} \left[ -\ln \left( \sum_{i=0}^{n-1} B_{n-1,i}(t_{(j)})\Delta z_{i} \right) + \ln \left( \sum_{i=0}^{n-1} B_{n-1,i}(t_{(j)})\Delta x_{i} \right) \right]. \label{eq:obj-nll}
\end{equation}
Since the MLE problem shares the same constraint structure as \eqref{eq:nonlinear} but with a different objective function, we follow the same restriction strategy developed in minimum error estimation, adapting the algorithmic framework as needed to preserve convergence guarantees.

\subsection{Restricted formulation and a primal gradient scheme}

We adopt the same restricted domain $\mathcal{C}$ from which we obtain the problem \[\min_{\mathbf{x},\mathbf{z},\mathbf{t}} \ \{\mathcal{L}(\mathbf{x},\mathbf{z}) \text{ s.t. } \eqref{constraint:sample}-\eqref{constraint:z_n=1}\},\] which is also suitable for off-the-shelf solvers. Furthermore, we apply the strategy of fixing the $x$-coordinates, and consequently $t_{(j)}$ for $j=1,...,m$, as described in Section \ref{sec:PGD}. We further suppose that the $x$-coordinates of the control points are distinct. Under these conditions, the second term of the objective function \eqref{eq:obj-nll} is well-defined and becomes constant. Thus, the optimization reduces to finding the $z$-coordinates that minimize the negative of the sum of logarithms of an affine function
\begin{equation*}
\min_{\mathbf{z}} \Bigl\{ \textstyle-\sum_{j=1}^{m} \ln \bigl(\sum_{i=0}^{n-1} B_{n-1,i}(t_{(j)}) \Delta z_{i} \bigr) \ : \ z_i \leq z_{i+1},\ i=0,\ldots,n-1; \ \ z_0 = 0,\ z_n = 1 \Bigr\}.
\end{equation*}
By introducing the change of variables $w_i = \Delta z_i$ for $i=0,...,n-1$, the feasible region is mapped to the simplex $\Delta^{n-1} = \{\mathbf{w}\in \mathbb{R}^n \mid\mathbf{1}^\top \mathbf{w} = 1,  \  \mathbf{w}\geq  \mathbf{0}\}$. This substitution naturally reduces the dimensionality of the decision space by one. Furthermore, by defining the coefficient vectors $\mathbf{a}_j\in \mathbb{R}^n$ for $j=1,...,m$ with components $(\mathbf{a}_{j})_i = B_{n-1,i}(t_{(j)})$ for $i=0,...,n-1$, the problem can be cast to \eqref{eq:simplexprob} as follows
\begin{equation} \label{eq:simplexprob}
\min_{\mathbf{w}} \Bigl\{ l(\mathbf{w}) =-\textstyle\sum_{j=1}^{m} \ln \left( \mathbf{a}_j^\top \mathbf{w} \right) \ : \ \mathbf{w} \in \Delta^{n-1} \Bigr\}.
\end{equation}
The objective function $l(\mathbf{w})$ is convex, as it is a composition of a convex function $-\ln(\cdot)$ with an affine map. The gradient and Hessian are given respectively by $\nabla l(\mathbf{w}) = - \sum_{j=1}^m\frac{\mathbf{a}_j}{\mathbf{a}_j^\top \mathbf{w}}$ and $\nabla^2l(\mathbf{w}) = \sum_{j=1}^m \frac{\mathbf{a}_j\mathbf{a}^\top_j}{(\mathbf{a}^\top_j \mathbf{w})^2}$. However, $l(\mathbf{w})$ is not $\beta$-smooth. For $t_{(j)}=0$ we have that $\mathbf{a}_j^\top = [1,0,...,0]$, so we can define the feasible vector $\mathbf{w} = [\frac{\epsilon}{n-1}, ... , \frac{\epsilon}{n-1}, 1 -\epsilon]^\top$ with $\epsilon>0$ such that $\frac{\mathbf{a}_j\mathbf{a}^\top_j}{(\mathbf{a}^\top_j \mathbf{w})^2} = \frac{\mathbf{a}_j\mathbf{a}^\top_j}{(\epsilon/(n-1))^2} \to \infty $ as $\epsilon \to 0$ (the same can be argued for $t_{(j)}=1$). Consequently, standard first-order methods like PGD lack theoretical convergence guarantees for this problem, motivating a different approach.
\begin{definition}[Relative smoothness (see, e.g., \citep{Hainao2018})]
Let $\mathcal{Q} \subseteq \mathbb{R}^n$ be a closed convex set with nonempty interior and $l, h : \mathcal{Q} \to \mathbb{R}$ be twice continuously differentiable on $\mathrm{int}(\mathcal{Q} )$.
We say that $l$ is \emph{$\beta$-smooth relative to $h$} on $\mathcal{Q} $ if there exists a constant $\beta > 0$ such that for all $\mathbf{w} \in \mathrm{int}(\mathcal{Q} )$ it holds that $\nabla^2 l(\mathbf{w}) \preceq \beta \nabla^2 h(\mathbf{w}),$
that is, $
\beta \nabla^2 h(\mathbf{w}) - \nabla^2 l(\mathbf{w})$ is positive semi-definite.
\end{definition}

Proposition \ref{prop:relsmooth} provides a reference function for our setting.
Specifically, it shows that $l(\mathbf{w})$ is $m$-smooth relative to the barrier-type function $h(\mathbf{w}) = -\sum_{i=0}^{n-1} \ln(w_i)$, which is convex and twice continuously differentiable on $\mathrm{int}(\Delta^{n-1})$.
\begin{proposition} \label{prop:relsmooth}
$l(\mathbf{w})$ is $m$-smooth with respect to $h(\mathbf{w})=-\sum_{i=0}^{n-1} \ln(w_i)$ on $\Delta^{n-1}$.
\end{proposition}
\begin{proof}{Proof.}
    See Appendix \ref{apx:proofs}.
\end{proof}

Therefore, problem \eqref{eq:simplexprob} satisfies the relative smoothness condition required by the framework of \citet{Hainao2018}. At each iteration $k$, their primal gradient scheme computes $\nabla l(\mathbf{w}^k)$ and updates the iterate by solving the subproblem $$\mathbf{w}^{k+1} \leftarrow
\displaystyle\argmin_{\mathbf{w} \in \Delta^{n-1}}
\left\{
l(\mathbf{w}^k)
+ \langle \nabla l(\mathbf{w}^k), \mathbf{w} - \mathbf{w}^k \rangle
+ m D_h(\mathbf{w}, \mathbf{w}^k)
\right\},$$ where $D_h(\mathbf{w}, \mathbf{w}^k)=h(\mathbf{w})-h(\mathbf{w}^k) - \langle \nabla h(\mathbf{w}^k), \mathbf{w} -\mathbf{w}^k \rangle$ denotes the Bregman distance induced by $h(\cdot)$. Since their algorithmic scheme only requires that the objective function be relatively smooth with respect to a convex reference function to achieve sublinear convergence, a condition met in Proposition \ref{prop:relsmooth}, we inherit the same convergence guarantees.

Nonetheless, we still need an efficient way to solve the subproblem. Removing the constant terms we get to the convex optimization problem $$\min_{\mathbf{w} \in \Delta^{n-1}}
\langle \nabla l(\mathbf{w}^k), \mathbf{w}\rangle
+ m(h(\mathbf{w}) - \langle \nabla h(\mathbf{w}^k), \mathbf{w}\rangle ).
$$
The first order optimality conditions are given by $\sum_{i=0}^{n-1} \frac{m}{\nabla l(\mathbf{w}^k)_i + m/w_i^k + \lambda} = 1$, $w_i> 0$, and $w_i= \frac{m}{\nabla l(\mathbf{w}^k)_i + m/w_i^k +\lambda}$ for $i=0,...,n-1$, where $\lambda\in \mathbb{R}$ is the Lagrange multiplier for the equality in the simplex constraint in \eqref{eq:simplexprob}. It only remains to find the value of $\lambda$. Observe that the valid domain for $\lambda$ is $(s, \infty)$ with $s = -\min_i \{\nabla l(\mathbf{w}^k)_i + m/w_i^k\}$ and $\phi(\lambda) =\sum_{i=0}^{n-1} \frac{m}{\nabla l(\mathbf{w}^k)_i + m/w_i^k + \lambda}$ is strictly decreasing with respect to $\lambda$. Since $\lim_{\lambda \downarrow s}\phi(\lambda) = +\infty$ and $\lim_{\lambda \to \infty}\phi(\lambda) = 0$ by continuity and strict monotonicity there exists a unique $\lambda^*$ such that $\phi(\lambda^*)=1$. Thus, we can use Newton's method or the bisection method to find $\lambda^*$ efficiently. Putting all of this together we arrive at the primal gradient scheme displayed in Algorithm \ref{alg:primal_gradient}.

\begin{algorithm}[htbp]
    \SetKwInput{KwParameters}{Parameters}
    \SetKwData{KwRequire}{Require}
    \SetKwData{KwEnsure}{Ensure}

    \caption{Primal gradient scheme}
    \label{alg:primal_gradient}
    \KwParameters{number of iterations $N$.}
    \KwIn{initial point $\mathbf{w}^0 \in \Delta^{n-1}$.}

    \For{$k = 0$ to $N-1$}{
        Compute $\nabla l(\mathbf{w}^k)$

        Compute $\lambda^*$ by solving $\phi(\lambda) = 1$

        $w^{k+1}_i \leftarrow \frac{m}{\nabla l(\mathbf{w}^k)_i +m/w_i^k +\lambda^*}$ \hspace{2cm} $i=0,...,n-1$
    }

    \Return $\mathbf{w}^N$
\end{algorithm}


\section{A note on convolutions} \label{sec:conv}
Let $X$ and $Y$ be two independent B\'ezier random variables of degree $n_X$ and $n_Y$ defined by the sets of control points $\mathcal{P}_X = \left\{\mathbf{p}_i^{X} = (x_i^{X},z_i^{X})^\top \right\}^{n_X}_{i=0}$ and $ \mathcal{P}_Y=\left\{\mathbf{p}_j^{Y} = (x_j^{Y},z_j^{Y})^\top \right\}_{j=0}^{n_Y},$ respectively, where $\mathcal{P}_X \in \mathcal{S}_{n_X}, \mathcal{P}_Y\in \mathcal{S}_{n_Y}$. Since $X$ and $Y$ are independent, the cdf of their convolution $Z = X + Y$ is given by
    \begin{equation}
        F_Z(z) = \int_{-\infty}^{\infty}f_X(x(t_X)) \left [ F_Y(z-x(t_X)) \right ]  \,dx(t_X).\label{eq:convolution1}
    \end{equation}
Evaluating \eqref{eq:convolution1} in closed form requires, in particular, an
explicit expression for the inner cdf as a function of the
integration variable. Recall that a Bézier cdf is defined only parametrically: to obtain $F_X(x)$ one must recover $t$ from $P_x(t;n_X,\mathbf x)=x$ and then evaluate
$P_z(t;n_X,\mathbf z)$, so $F_X=P_z\circ P_x^{-1}$. The problematic term is the inverse $P_x^{-1}$, since, in general, no elementary formula exists for it (not solvable by radicals) when $n_X\geq 5$ by the Abel's Theorem on the insolvability of the general quintic \citep{DummitFoote2004}. This already rules out a closed form for the constituent cdf $F_Y$ appearing in \eqref{eq:convolution1}. Thus, we cannot expect arbitrary Bézier distributions to admit a tractable analytical convolution.

We therefore approximate the convolution by Monte Carlo. Given $M$ constituent B\'ezier cdfs $\{F_j\}_{j=1}^M$ and a sample size $N$, the procedure repeats, for $i=1,\ldots,N$, drawing independent uniform variates $u_1^i,\ldots,u_M^i\sim\mathcal{U}(0,1)$ and forming an aggregate sample $s_i = \sum_{j=1}^{M} F_j^{-1}(u_j^i)$ by inverse-transform sampling of each variable. It then fits a new B\'ezier distribution to the resulting samples $\{s_i\}_{i=1}^N$ using the methods of the preceding sections, returning the approximate aggregate cdf $\hat{F}_{\mathrm{agg}}$.

\section{The \texttt{bezierv} Python package} \label{sec:package}

The \texttt{bezierv} Python package provides: (i) the B\'ezier random variable (Section~\ref{sec:cdf}), (ii) estimation routines for sample data (Sections~\ref{sec:fitting} and \ref{sec:mle}), (iii) convolution utilities (Section~\ref{sec:conv}), (iv) random variate generation, and (v) an interactive fitting component. The implementation is in Python and builds on \texttt{NumPy} \citep{Harris2020}. The package is distributed via the Python Package Index (PyPI) and can be installed with \verb|pip install bezierv|. Source code and the issue tracker are hosted on GitHub~\href{https://github.com/EstebanLeiva/bezierv}{EstebanLeiva/bezierv}, and the API documentation is available at \href{https://estebanleiva.github.io/bezierv/}{estebanleiva.github.io/bezierv/}. The codebase includes unit tests and continuous-integration workflows to promote correctness and reproducibility across supported Python versions: 3.10, 3.11, 3.12, 3.13.

\subsection{Package structure}
The package has three core classes, \texttt{Bezierv}, \texttt{DistFit}, and \texttt{Convolver}, that mirror the functionalities outlined above. \texttt{Bezierv} encapsulates a B\'ezier random variable, storing its support and control-point parameterization, and exposes evaluation utilities such as random sampling, quantile evaluation, moment computation, and visualization. \texttt{DistFit} fits a B\'ezier distribution to empirical data and returns a \texttt{Bezierv} instance by providing an interface to access any of  the estimation algorithms presented in Section~\ref{sec:fitting} and \ref{sec:mle}. \texttt{Convolver} takes a collection of \texttt{Bezierv} objects and approximates the distribution of their sum via Monte Carlo simulation. Figure~\ref{fig:package_structure} depicts the high-level structure of the package and the relationships among these three classes.

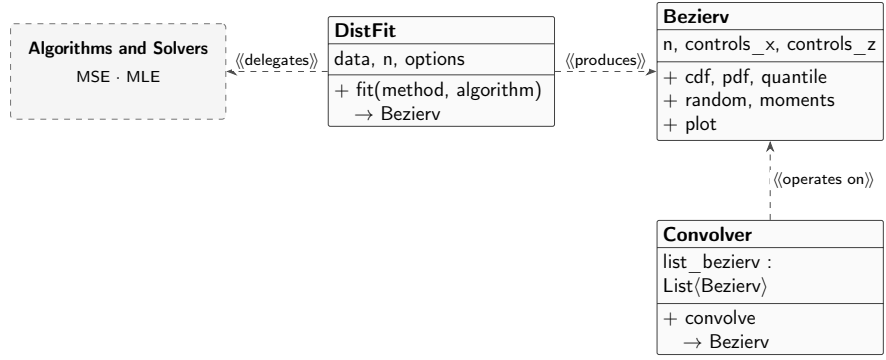
\begin{figure}[htbp]
    \centering
    \resizebox{0.7\textwidth}{!}{%
    \begin{tikzpicture}[
            font=\small\sffamily,
            >={Stealth[length=2mm,width=1.5mm]},
            class/.style={
                rectangle split, rectangle split parts=3,
                rectangle split draw splits=true,
                draw=black!75, line width=0.5pt,
                rounded corners=1pt,
                text width=38mm, align=left,
                inner sep=3pt,
                fill=black!2,
            },
            head/.style={font=\small\sffamily\bfseries},
            pkg/.style={
                draw=black!55, line width=0.4pt, dashed,
                rounded corners=2pt, inner sep=5pt,
                fill=black!4,
            },
            dep/.style={->, draw=black!70, line width=0.5pt, dashed},
            edgelab/.style={font=\scriptsize\sffamily, fill=white, inner sep=1pt},
        ]

        \node[font=\footnotesize\sffamily, align=center] (solvers) {%
            \textbf{Algorithms and Solvers}\\[3pt]
            \scriptsize MSE\ $\cdot$\ MLE\\[-1pt]
        };
        \begin{scope}[on background layer]
            \node[pkg, fit=(solvers)] (solversbox) {};
        \end{scope}

        \node[class, right=18mm of solversbox] (distfit) {
            \nodepart[head]{one}DistFit
            \nodepart{two}
                data,\ n,\ options
            \nodepart{three}
                \(+\) fit(method, algorithm)\\
                \quad$\rightarrow$ Bezierv
        };

        \node[class, right=18mm of distfit] (bezierv) {
            \nodepart[head]{one}Bezierv
            \nodepart{two}
                n,\ controls\_x,\ controls\_z
            \nodepart{three}
                \(+\) cdf,\ pdf,\ quantile\\
                \(+\) random,\ moments\\
                \(+\) plot
        };

        \node[class, below=14mm of bezierv] (conv) {
            \nodepart[head]{one}Convolver
            \nodepart{two}
                list\_bezierv : List$\langle$Bezierv$\rangle$
            \nodepart{three}
                \(+\) convolve\\
                \quad$\rightarrow$ Bezierv
        };

        \draw[dep] (distfit.west) -- node[edgelab,above]{$\langle\!\langle$delegates$\rangle\!\rangle$}
            (solversbox.east);

        \draw[dep] (distfit.east) -- node[edgelab,above]{$\langle\!\langle$produces$\rangle\!\rangle$}
            (bezierv.west);

        \draw[dep] (conv.north) -- node[edgelab,right]{$\langle\!\langle$operates on$\rangle\!\rangle$}
            (bezierv.south);

    \end{tikzpicture}}
    \caption{High-level package structure.}
    \label{fig:package_structure}
\end{figure}

\subsection{Input modeling usage example}
\setcounter{listing}{1} 
\setcounter{figure}{2} 
We walk through a complete input-modeling workflow on a synthetic manufacturing dataset. An injection-molding machine fills four mold cavities whose finished parts drop into a single bin, so the producing cavity is unknown. All cavities target a $10$~cm diameter, but wear makes them drift: cavities~1 and~2 run under-sized, cavity~3 stays calibrated, and cavity~4 runs over-sized, so the modeler receives a single pooled, multimodal sample of diameters to turn into a distribution for a downstream simulation. Figure~\ref{lst:snippet} carries out this workflow, described step by step below.
\begin{listing}[htbp]
\caption{Fitting and using a B\'ezier distribution for the molding data.}
\label{lst:snippet}
\begin{lstlisting}
import numpy as np
from bezierv import DistFit

# Step 1: pooled diameter measurements collected from the bin
data = ...                        # 1-D array, one diameter per part

# Step 2: fit a Bézier distribution by minimum squared error
fitter = DistFit(data, n=15)      # n = polynomial degree
bz, mse = fitter.fit(method='mse', algorithm='solver')

# Step 3: assess the fit
bz.plot_cdf(data)                 # fitted CDF over the empirical CDF
bz.plot_pdf(data)                 # fitted PDF over the data histogram
mean, variance = bz.mean(), bz.variance()
q10, q25 = bz.quantile(0.10), bz.quantile(0.25)
q75, q90 = bz.quantile(0.75), bz.quantile(0.90)

# Step 4 (manual): refine bz in the interactive component — see text

# Step 5: draw variates from the refined distribution
variates = bz.random(10_000, rng=42)
\end{lstlisting}
\end{listing}

\emph{Step 1 (data).} The input is a one-dimensional array of pooled diameters; here we generate it synthetically as a mixture of four normals (one per cavity) for reproducibility, though in practice it is collected directly from the process.

\emph{Step 2 (fit).} A single \texttt{DistFit}/\texttt{fit} call returns a fitted \texttt{Bezierv} object \texttt{bz} and its mean squared error; we use minimum-error estimation (\texttt{method='mse'}) with IPOPT (\texttt{algorithm='solver'}) and a curve of degree $15$, expressive enough to capture the multimodality.

\emph{Step 3 (assess).} The fit is checked graphically and numerically: \texttt{bz.plot\_cdf} and \texttt{bz.plot\_pdf} overlay the B\'ezier cdf and pdf. Figure~\ref{fig:multimodal_fit} shows the fitted pdf and cdf with their control points---degree $15$ means sixteen control points, several so close as to be visually indistinguishable---reproducing the three process modes: an under-sized peak on the left, the dominant nominal $10$~cm peak in the center, and cavity~4 on the right. Table~\ref{fig:descriptive_stats_multimodal} reports descriptive statistics from \texttt{bz.mean}, \texttt{bz.variance}, and \texttt{bz.quantile}, which closely match their sample counterparts.
\begin{figure}[htbp]
    \centering
    \begin{minipage}{0.42\textwidth}
        \centering
        \resizebox{\textwidth}{!}{%
        \input{multimodal.pgf}}
        \captionof{figure}{Fitted B\'ezier PDF and CDF with control points.}
        \label{fig:multimodal_fit}
    \end{minipage}
    \hfill
    \begin{minipage}{0.42\textwidth}
        \centering
        {\scriptsize\setlength{\tabcolsep}{3pt}
        \begin{tabular}{lrr}
        \toprule
        Statistic & \multicolumn{1}{c}{Sample} & \multicolumn{1}{c}{B\'ezier fit} \\
        \midrule
        $q_{0.10}$ & 9.5  & 9.5  \\
        $q_{0.25}$ & 9.7  & 9.7  \\
        $q_{0.75}$ & 10.2 & 10.2 \\
        $q_{0.90}$ & 10.5 & 10.5 \\
        Mean       & 10.0 & 10.0 \\
        Variance   & 0.1  & 0.1  \\
        \bottomrule
        \end{tabular}}
        \captionof{table}{Descriptive statistics.}
        \label{fig:descriptive_stats_multimodal}
    \end{minipage}
\end{figure}

\emph{Step 4 (refine with domain knowledge).} The automatic fit reflects only the pooled sample, but the modeler knows a scheduled maintenance on the under-sized cavities will return them to the nominal diameter, moving the mass in the left peak to the central, on-target one. Using the interactive component, the modeler drags the control points to flatten the left peak and concentrate mass at $10$~cm. Figures~\ref{fig:ui} and~\ref{fig:multimodal_manual} show the interactive component and the hand-adjusted distribution, in which the central peak now dominates.
\begin{figure}[htbp]
    \centering
    \begin{subfigure}{0.27\textwidth}
        \centering
        \includegraphics[width=\textwidth]{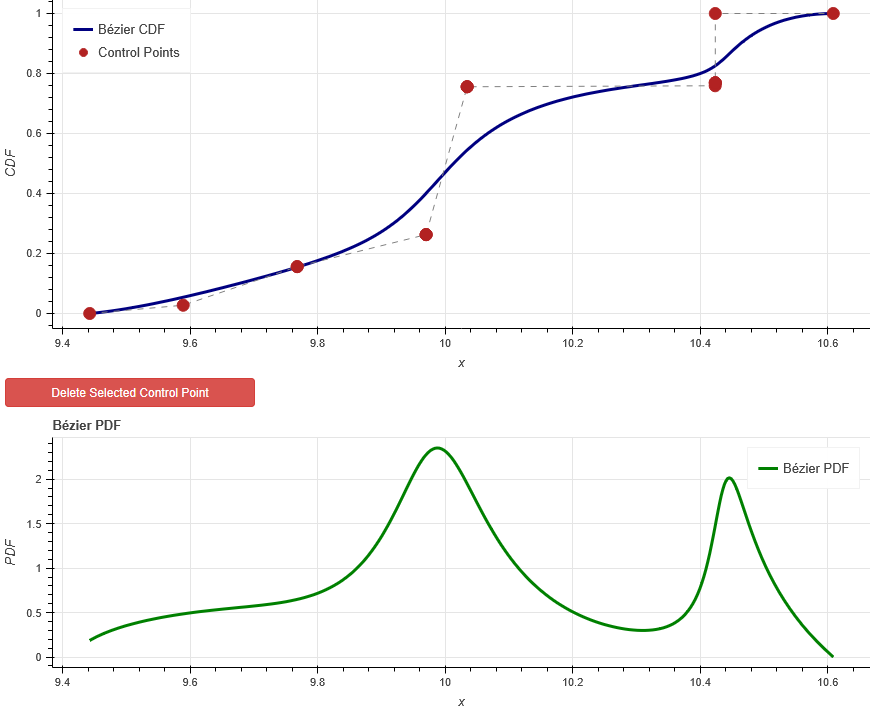}
        \caption{Interactive fitting component.}
        \label{fig:ui}
    \end{subfigure}
    \hfill
    \begin{subfigure}{0.32\textwidth}
        \centering
        \resizebox{\textwidth}{!}{%
        \input{multimodal_manual.pgf}}
        \caption{Hand-adjusted B\'ezier r.v.}
        \label{fig:multimodal_manual}
    \end{subfigure}
    \hfill
    \begin{subfigure}{0.32\textwidth}
        \centering
        \resizebox{\textwidth}{!}{%
        \input{multimodal_manual_hist.pgf}}
        \caption{Variates from the adjusted r.v.}
        \label{fig:sample}
    \end{subfigure}
    \caption{Manual editing and random sampling with \texttt{bezierv}.}
    \label{fig:both}
\end{figure}

\emph{Step 5 (generate variates).} A single \texttt{bz.random} call then draws independent variates by inverse-transform sampling, ready to feed a simulation engine. Figure~\ref{fig:sample} shows the histogram of $10{,}000$ draws, which reproduces the refined shape with its dominant central peak.

\section{Numerical experiments} \label{sec:experiments}
Given the documented advantages of B\'ezier distributions in the simulation literature \citep{Kuhl2010}, we now turn to computational experiments that confirm these findings and, more importantly, demonstrate that our fitting methods are significantly faster than existing approaches. Specifically, we evaluate the proposed B\'ezier fitting algorithms against the derivative-free methods of \citet{Wagner1993} and compare the B\'ezier fit with the generalized beta distribution and the Johnson translation system. All experiments were conducted on a machine with a 12th Gen Intel(R) Core(TM) i9-12900H 2.50 GHz processor and 32 GB of RAM and are available at \href{https://github.com/EstebanLeiva/bezierv-experiments}{EstebanLeiva/bezierv-experiments} for reproducibility. The evaluation proceeds in two stages: first, synthetic datasets are used to examine the ability of each method to recover a broad range of distributional shapes under controlled conditions; second, real-world transportation data is used to assess robustness and modeling flexibility in practical settings.

Recall that we derived two tractable approximations of the B\'ezier minimum error estimation problem \eqref{eq:nonlinear}: the restricted formulation \eqref{eq:restricted}, which enforces monotonicity on the control points and is amenable to IPOPT \citep{Wachter2006}, and the convex quadratic program \eqref{eq:convexQP}, obtained by further fixing the $x$-coordinates, which we solve with PGD. The same restrictions apply to the MLE problem, for which we developed the primal gradient scheme with identical convergence guarantees; we therefore focus the algorithmic comparison on minimum error estimation, as the analysis carries over directly. Since the MSE objective \eqref{obj:sumofsquares} is bounded below by zero, the quality of these restrictions can be directly assessed from the attained objective values. As we demonstrate below, both formulations consistently achieve average MSE values on the order of $10^{-2}$ or smaller, confirming that the restrictions incur negligible loss in fit quality despite providing only upper bounds on the optimal value of the original problem.

\subsection{Synthetic data}
We test the B\'ezier-based fitting procedure on probability distributions commonly used in simulation input modeling~\citep{Law_2015}, as well as on bimodal and trimodal Gaussian mixtures designed to capture multimodal behavior.
For each distribution family, we generate 200 instances with randomly sampled parameters (e.g., mean and variance), and draw 1{,}000 observations from each instance. For each dataset, we fit a Bézier distribution of degree $n=10$ using the proposed primal gradient scheme for MLE, with a maximum of 1{,}000 iterations and an early stopping criterion based on $\frac{\|\mathbf{w}^{k+1} - \mathbf{w}^k\|}{\max(1,\|\mathbf{w}^k\| )} \leq 10^{-3}$. We report the computational time (ms) and the resulting negative log-likelihood (NLL), and compare the results in Table~\ref{tab:syntheticmle} with those obtained from the generalized beta and Johnson translation systems fitted via \texttt{scipy.stats}~\citep{Scipy2020}. Note that we restrict the comparison with these two distribution families to maximum likelihood estimation, as it is a standard method for fitting rich distributions to sample data in the simulation literature \citep{Law_2015} and the one implemented in the software package used for benchmarking.
\begin{table}[htbp]
    \centering
    {\scriptsize\setlength{\tabcolsep}{3pt}
    \begin{tabular}{l rr rr rr rr rr rr}
        \toprule
        & \multicolumn{4}{c}{B\'ezier} & \multicolumn{4}{c}{Beta} & \multicolumn{4}{c}{Johnson} \\
        \cmidrule(lr){2-5} \cmidrule(lr){6-9} \cmidrule(lr){10-13}
        & \multicolumn{2}{c}{time (ms)} & \multicolumn{2}{c}{NLL} & \multicolumn{2}{c}{time (ms)} & \multicolumn{2}{c}{NLL} & \multicolumn{2}{c}{time (ms)} & \multicolumn{2}{c}{NLL} \\
        \cmidrule(lr){2-3} \cmidrule(lr){4-5} \cmidrule(lr){6-7} \cmidrule(lr){8-9} \cmidrule(lr){10-11} \cmidrule(lr){12-13}
        dist & mean & std & mean & std & mean & std & mean & std & mean & std & mean & std \\
        \midrule
        Beta               & 2.9 & 1.4 & 1836.9 & 722.7 & 15.2 & 9.9 & 1830.6 & 725.1 & 44.1 & 11.7 & 1889.5 & 716.9 \\
        Bimodal Gaussian   & 34.0 & 17.9 & \textbf{2570.8} & 269.3 & 27.6 & 11.2 & 2931.2 & 398.7 & 35.6 & 15.9 & 2875.4 & 334.8 \\
        Exponential        & 4.5 & 2.8 & 2420.9 & 686.2 & 33.6 & 10.3 & 2436.0 & 678.9 & 43.1 & 10.5 & 2448.7 & 686.3 \\
        Gamma              & 6.4 & 5.6 & 2832.7 & 749.0 & 25.1 & 11.6 & 2823.6 & 748.0 & 43.2 & 10.9 & 2829.7 & 740.5 \\
        Log Logistic       & 32.6 & 20.9 & 1481.3 & 718.6 & 33.8 & 10.4 & 1535.9 & 737.1 & 19.0 & 10.1 & 1468.2 & 717.7 \\
        Lognormal          & 15.8 & 12.7 & 1726.4 & 908.5 & 31.6 & 11.1 & 1772.8 & 946.5 & 37.9 & 12.7 & 1714.7 & 907.6 \\
        Normal             & 11.8 & 8.7 & 2103.4 & 549.1 & 19.4 & 11.1 & 2094.0 & 549.0 & 33.7 & 14.7 & 2094.1 & 549.0 \\
        Triangular         & 2.3 & 0.7 & 2116.5 & 590.9 & 12.2 & 5.2 & 2119.4 & 590.4 & 38.7 & 8.8 & 2148.8 & 591.9 \\
        Trimodal Gaussian  & 29.5 & 18.2 & \textbf{2431.0} & 391.6 & 27.2 & 11.0 & 2825.2 & 352.6 & 24.6 & 14.8 & 2683.9 & 443.3 \\
        Uniform            & 0.5 & 0.1 & 2231.2 & 694.9 & 25.9 & 10.8 & 2237.4 & 698.7 & 40.0 & 11.1 & 2410.4 & 694.7 \\
        Weibull            & 5.5 & 4.0 & 1271.4 & 685.2 & 15.1 & 9.0 & 1268.9 & 693.2 & 38.8 & 9.2 & 1275.2 & 691.5 \\
        \midrule
        Average            & \textbf{13.2} & 16.6 & \textbf{2092.9} & \textbf{796.0} & 24.2 & 12.6 & 2170.2 & 845.9 & 36.2 & 14.2 & 2167.1 & 836.2 \\
        \bottomrule
    \end{tabular}}
    \caption{MLE comparison (synthetic data).}
    \label{tab:syntheticmle}
\end{table}

The results underscore the flexibility of the B\'ezier distribution. Its average NLL is comparable to the competing families on unimodal distributions but substantially better on multimodal data (differences of several hundred in NLL). Its NLL standard deviation ($796.0$) is likewise the lowest of the three, signaling more uniform fit quality. However, when the underlying distribution is standard the extra flexibility is unnecessary, and the classical families perform comparably or, as for the Beta, marginally better, as expected when the true model lies in the parametric class. In terms of runtime the B\'ezier fit is the fastest, averaging about half the generalized beta and a third the Johnson translation system, which is striking given its nine free parameters (the $z$-coordinates of the $n-1$ interior control points) against four for each competitor. The primal gradient scheme thus delivers this richer parameterization at lower cost, eliminating the speed-for-flexibility trade-off that previously discouraged adoption.
\begin{table}[htbp]
    \centering
    {\scriptsize\setlength{\tabcolsep}{3pt}
    \begin{tabular}{l rr rr c rr rr c rr rr c}
        \toprule
        & \multicolumn{5}{c}{PGD} & \multicolumn{5}{c}{IPOPT} & \multicolumn{5}{c}{NM} \\
        \cmidrule(lr){2-6} \cmidrule(lr){7-11} \cmidrule(lr){12-16}
        & \multicolumn{2}{c}{time (ms)} & \multicolumn{2}{c}{MSE} & & \multicolumn{2}{c}{time (ms)} & \multicolumn{2}{c}{MSE} & & \multicolumn{2}{c}{time (ms)} & \multicolumn{2}{c}{MSE} & \\
        \cmidrule(lr){2-3} \cmidrule(lr){4-5} \cmidrule(lr){7-8} \cmidrule(lr){9-10} \cmidrule(lr){12-13} \cmidrule(lr){14-15}
        dist & mean & std & mean & std & fail & mean & std & mean & std & fail & mean & std & mean & std & fail \\
        \midrule
        Beta               & 1.1 & 0.2 & 2.7e-5 & 9.0e-6 & 0.0\% & 2937.7 & 3044.2 & 1.2e-5 & 3.6e-6 & 0.5\% & 13270.4 & 989.8 & 2.4e-4 & 2.9e-4 & 0.0\% \\
        Bimodal Gaussian   & 1.9 & 0.6 & 5.6e-4 & 5.9e-4 & 0.0\% & 2181.3 & 466.2 & 9.2e-5 & 1.4e-4 & 0.0\% & 13503.7 & 885.6 & 2.0e-3 & 1.9e-3 & 0.0\% \\
        Exponential        & 1.3 & 0.3 & 3.5e-5 & 1.2e-5 & 0.0\% & 2959.0 & 1002.7 & 1.2e-5 & 3.6e-6 & 0.5\% & 14973.3 & 738.8 & 1.8e-4 & 2.8e-4 & 0.0\% \\
        Gamma              & 1.4 & 0.3 & 3.5e-5 & 1.4e-5 & 0.0\% & 2573.3 & 671.4 & 1.2e-5 & 4.1e-6 & 0.0\% & 14992.7 & 790.3 & 1.8e-4 & 9.8e-5 & 0.0\% \\
        Log Logistic       & 1.9 & 0.6 & 8.9e-5 & 1.2e-4 & 0.0\% & 2295.8 & 419.7 & 1.3e-5 & 4.7e-6 & 0.0\% & 15734.1 & 1236.3 & 1.7e-3 & 1.7e-3 & 0.0\% \\
        Lognormal          & 1.5 & 0.3 & 4.3e-5 & 2.6e-5 & 0.0\% & 2422.5 & 414.4 & 1.3e-5 & 4.1e-6 & 0.0\% & 15663.3 & 1345.5 & 5.6e-4 & 6.1e-4 & 0.0\% \\
        Normal             & 1.5 & 0.4 & 3.8e-5 & 1.1e-5 & 0.0\% & 2314.7 & 487.3 & 1.2e-5 & 3.7e-6 & 0.0\% & 13528.0 & 901.0 & 8.6e-4 & 5.7e-4 & 0.0\% \\
        Triangular         & 1.1 & 0.2 & 2.5e-5 & 7.5e-6 & 0.0\% & 2325.5 & 498.1 & 1.1e-5 & 3.2e-6 & 0.5\% & 11948.0 & 722.3 & 9.4e-5 & 8.4e-5 & 0.0\% \\
        Trimodal Gaussian  & 1.9 & 0.6 & 9.1e-4 & 9.6e-4 & 0.0\% & 2294.9 & 597.7 & 3.2e-4 & 4.0e-4 & 0.0\% & 12506.9 & 1141.1 & 3.1e-3 & 2.6e-3 & 0.0\% \\
        Uniform            & 0.9 & 0.1 & 2.2e-5 & 7.4e-6 & 0.0\% & 2768.3 & 1213.1 & 1.1e-5 & 3.4e-6 & 0.0\% & 10527.4 & 752.8 & 2.3e-5 & 8.8e-6 & 0.0\% \\
        Weibull            & 1.2 & 0.2 & 3.4e-5 & 1.3e-5 & 0.0\% & 2253.7 & 508.1 & 1.2e-5 & 3.6e-6 & 0.0\% & 12577.1 & 1111.9 & 3.6e-4 & 3.8e-4 & 0.0\% \\
        \midrule
        Average            & \textbf{1.4} & \textbf{0.5} & 1.7e-4 & 4.4e-4 & \textbf{0.0\%} & 2483.9 & 1150.4 & \textbf{4.7e-5} & \textbf{1.6e-4} & 0.1\% & 13565.9 & 1854.6 & 8.4e-4 & 1.5e-3 & 0.0\% \\
        \bottomrule
    \end{tabular}}
    \caption{MSE comparison (synthetic data).}
    \label{tab:syntheticmse}
\end{table}

Under the same setup, we compare PGD (with stopping criterion $\frac{\|\mathbf{z}^{k+1} - \mathbf{z}^k\|}{\max(1,\|\mathbf{z}^k\|)} \leq 10^{-3}$), IPOPT, and the Nelder--Mead (NM) algorithm---the original method for B\'ezier fitting \citep{Wagner1993}, as implemented in \citet{Scipy2020}---on minimum error estimation. PGD and NM are capped at 200 iterations and IPOPT at 60 seconds of CPU time. Table~\ref{tab:syntheticmse} reports MSE and running times. IPOPT attains the lowest average MSE ($10^{-5}$) ahead of PGD and NM ($10^{-4}$), as expected since \eqref{eq:restricted} optimizes over both $\mathbf{x}$ and $\mathbf{z}$, whereas \eqref{eq:convexQP} fixes $\mathbf{x}$. However, this roughly one-order accuracy gap may not be practically significant, while the computational trade-off is dramatic: PGD is by far the fastest (1.4\,ms per fit, lowest variability), against over $1{,}700\times$ that effort for IPOPT and $9{,}000\times$ for NM. Moreover, IPOPT is not fully reliable, failing (with errors or crashes) on about $0.1\%$ of instances. Since the MLE primal gradient scheme operates on the same restricted set with analogous convergence guarantees, these advantages carry over directly.

\subsection{Transportation data} \label{sec:transportationdata}
The Chicago road network dataset curated by \citet{Goerigk2025} comprises 3.89 million speed records from real traffic data. The network is represented as a directed graph with 538 nodes and 1308 arcs, where each arc is associated with a collection of travel time observations. After excluding arcs containing only a single observation, we obtained a reduced graph of 1,086 arcs. To reflect different levels of temporal information available in routing applications, we partition the data into three time regimes: morning rush hour (6:00–10:00), daytime (6:00–18:00), and full day (0:00–23:59). These scenarios correspond to settings in which a traveler knows they will traverse the network during peak hours, during the broader daytime period, or has no specific departure-time information and must therefore consider the entire daily distribution.

Table~\ref{tab:realdist} reports, for each time regime, the distribution of the
number of travel-time observations per arc (mean, standard deviation, minimum,
and maximum), together with three arc-level travel-time statistics: the mean travel time (in minutes), the standard deviation, and the average range. As expected, as the window widens from the morning rush hour to the full day, the number of observations per arc and the travel-time range increase.

\begin{table}[htbp]
    \centering
    {\scriptsize\setlength{\tabcolsep}{3pt}
    \begin{tabular}{ll rrr}
        \toprule
        & statistic & \multicolumn{1}{c}{Rush hour} & \multicolumn{1}{c}{Daytime} & \multicolumn{1}{c}{Full-day} \\
        \midrule
        \multirow{4}{*}{\# obs.} & mean  & 550.8 & 1652.3 & 3274.2 \\
                                 & std   & 68.4 & 205.3 & 406.8 \\
                                 & min   & 542.0 & 1626.0 & 3222.0 \\
                                 & max   & 1084.0 & 3252.0 & 6444.0 \\
        \midrule
        \multirow{3}{*}{travel time} & mean  & 1.44 & 1.48 & 1.44 \\
                                     & std   & 0.36 & 0.41 & 0.40 \\
                                     & range & 3.31 & 5.53 & 6.27 \\
        \bottomrule
    \end{tabular}}
    \caption{Distribution of the number of observations per arc and arc-level travel time statistics across different time regimes.}
    \label{tab:realdist}
\end{table}

Using this dataset, we fit a B\'ezier distribution to each arc via MLE with the proposed primal gradient scheme, increasing the degree $n$ of the B\'ezier distribution to accommodate the greater variety of distributional shapes arising from wider time windows. Table~\ref{tab:realmle} reports the resulting NLL values and computational times. On average, the Johnson translation system and the generalized beta attain lower mean NLL than the B\'ezier fit; however, this average masks lack of robustness. The standard deviation of the NLL is roughly 2.4 and 5.4 times larger for the beta and Johnson models, respectively, than for the B\'ezier distribution. This indicates that while the beta and Johnson distributions fit well on instances that align with their structural assumptions, they fit poorly on others, whereas the B\'ezier distribution delivers consistently comparable fits across a far wider range of shapes. The same robustness is visible per regime: the dispersion of the competing models grows sharply as the time window widens and the Johnson translation system reaches a NLL standard deviation above $14{,}000$ on the full-day data, while the B\'ezier fit remains comparatively stable. In terms of computational time, Johnson is fastest on average, but the B\'ezier fit trails by only about 5\,ms; moreover, on the largest instances (full-day) the B\'ezier distribution is actually the fastest of the three, despite using the highest number of parameters.

\begin{table}[htbp]
    \centering
    {\scriptsize\setlength{\tabcolsep}{3pt}
    \begin{tabular}{lc rr rr rr rr rr rr}
        \toprule
        & & \multicolumn{4}{c}{B\'ezier} & \multicolumn{4}{c}{Beta} & \multicolumn{4}{c}{Johnson} \\
        \cmidrule(lr){3-6} \cmidrule(lr){7-10} \cmidrule(lr){11-14}
        & & \multicolumn{2}{c}{time (ms)} & \multicolumn{2}{c}{NLL} & \multicolumn{2}{c}{time (ms)} & \multicolumn{2}{c}{NLL} & \multicolumn{2}{c}{time (ms)} & \multicolumn{2}{c}{NLL} \\
        \cmidrule(lr){3-4} \cmidrule(lr){5-6} \cmidrule(lr){7-8} \cmidrule(lr){9-10} \cmidrule(lr){11-12} \cmidrule(lr){13-14}
        time frame & $n$ & \multicolumn{1}{c}{mean} & \multicolumn{1}{c}{std} & \multicolumn{1}{c}{mean} & \multicolumn{1}{c}{std} & \multicolumn{1}{c}{mean} & \multicolumn{1}{c}{std} & \multicolumn{1}{c}{mean} & \multicolumn{1}{c}{std} & \multicolumn{1}{c}{mean} & \multicolumn{1}{c}{std} & \multicolumn{1}{c}{mean} & \multicolumn{1}{c}{std} \\
        \midrule
        Rush hour    & 10 & 48.1 & 37.7 & $-$2226.7 & 431.0 & 42.7 & 19.2 & $-$2362.0 & 1153.5 & 36.5 & 22.7 & $-$2745.8 & 2178.3 \\
        Daytime      & 15 & 54.3 & 34.1 & $-$6351.9 & 1241.8 & 50.5 & 20.9 & $-$6806.5 & 3140.5 & 45.2 & 26.1 & $-$8103.6 & 6212.3 \\
        Full-day     & 20 & \textbf{52.5} & 25.9 & $-$12734.4 & 2473.4 & 60.2 & 11.3 & $-$13553.3 & 5691.8 & 58.0 & 26.9 & $-$17129.1 & 14038.2 \\
        \midrule
        \multicolumn{2}{l}{average} & 51.6 & 32.6 & $-$7104.3 & \textbf{1382.1} & 51.1 & 17.1 & $-$7573.9 & 3328.6 & \textbf{46.5} & 25.2 & $-$9326.2 & 7476.2 \\
        \bottomrule
    \end{tabular}}
    \caption{MLE comparison (transportation data).}
    \label{tab:realmle}
\end{table}

Within the same setting, we benchmark the minimum error estimation algorithms in Table~\ref{tab:MDEtab}. PGD runs three to four orders of magnitude faster than IPOPT and NM with lower run-time variability, whereas IPOPT attains the lowest MSE. The methods also scale very differently: from $n=10$ to $n=20$, PGD's average runtime rises only from 1.2 to 4.8\,ms, whereas IPOPT's grows by more than an order of magnitude (1{,}056 to 24{,}044\,ms). IPOPT's accuracy comes at the cost of reliability---it failed on $20.1\%$ of instances, versus $0.2\%$ for NM and none for PGD---and since failures are excluded from the averages, its MSE reflects only the better-behaved instances it could solve and is not directly comparable to PGD's over all instances. PGD is therefore the more robust choice, which is especially valuable when the data distribution is unknown a priori.
\begin{table}[htbp]
    \centering
    {\scriptsize\setlength{\tabcolsep}{3pt}
    \begin{tabular}{lc rr rr c rr rr c rr rr c}
        \toprule
        & & \multicolumn{5}{c}{PGD} & \multicolumn{5}{c}{IPOPT} & \multicolumn{5}{c}{NM} \\
        \cmidrule(lr){3-7} \cmidrule(lr){8-12} \cmidrule(lr){13-17}
        & & \multicolumn{2}{c}{time (ms)} & \multicolumn{2}{c}{MSE} & & \multicolumn{2}{c}{time (ms)} & \multicolumn{2}{c}{MSE} & & \multicolumn{2}{c}{time (ms)} & \multicolumn{2}{c}{MSE} & \\
        \cmidrule(lr){3-4} \cmidrule(lr){5-6} \cmidrule(lr){8-9} \cmidrule(lr){10-11} \cmidrule(lr){13-14} \cmidrule(lr){15-16}
        time frame & $n$ & \multicolumn{1}{c}{mean} & \multicolumn{1}{c}{std} & \multicolumn{1}{c}{mean} & \multicolumn{1}{c}{std} & \multicolumn{1}{c}{fail} & \multicolumn{1}{c}{mean} & \multicolumn{1}{c}{std} & \multicolumn{1}{c}{mean} & \multicolumn{1}{c}{std} & \multicolumn{1}{c}{fail} & \multicolumn{1}{c}{mean} & \multicolumn{1}{c}{std} & \multicolumn{1}{c}{mean} & \multicolumn{1}{c}{std} & \multicolumn{1}{c}{fail} \\
        \midrule
        Rush hour    & 10 & 1.2 & 0.5 & 0.0131 & 0.0570 & 0.0\% & 1056.0 & 1376.7 & 0.0014 & 0.0257 & 12.3\% & 7677.3 & 1350.4 & 0.0621 & 0.1058 & 0.4\% \\
        Daytime      & 15 & 2.6 & 0.9 & 0.0159 & 0.0538 & 0.0\% & 7362.9 & 7707.7 & 0.0016 & 0.0326 & 17.3\% & 25471.9 & 5163.3 & 0.1659 & 0.1257 & 0.1\% \\
        Full-day     & 20 & 4.8 & 1.4 & 0.0147 & 0.0502 & 0.0\% & 24043.5 & 25524.9 & 0.0011 & 0.0165 & 30.6\% & 53841.9 & 12276.6 & 0.2043 & 0.1323 & 0.0\% \\
        \midrule
        \multicolumn{2}{l}{average} & \textbf{2.9} & \textbf{0.9} & 0.0146 & 0.0537 & \textbf{0.0\%} & 10820.8 & 11536.4 & \textbf{0.0013} & \textbf{0.0249} & 20.1\% & 28997.0 & 6263.4 & 0.1441 & 0.1213 & 0.2\% \\
        \bottomrule
    \end{tabular}}
    \caption{MSE comparison (transportation data).}
    \label{tab:MDEtab}
\end{table}

\subsection{Effect of the restriction on fit quality} \label{sec:fitquality}
To validate the approximation guarantee of Theorem~\ref{theorem} in a
controlled setting, we design an experiment based on the cdf of
Example~\ref{counterexample}, hereafter denoted $F(t;\mathbf{x},\mathbf{z})$, whose control points violate the monotonicity condition and therefore lie in
$\mathcal{S}\setminus\mathcal{C}$. For each value of $n$, we draw
$m = 20\,n^2$ samples from $F(t;\mathbf{x},\mathbf{z})$ and fit a $n$-th degree B\'ezier distribution
$\hat{F}_n(t;\hat{\mathbf{x}},\hat{\mathbf{z}})$ by solving formulation~\eqref{eq:convexQP} with PGD, capping the
number of iterations at $\min(500,\, 12n)$. We repeat this procedure over
$10$ independent replications and let $n$ range over a logarithmically
spaced grid from $3$ to $100$.
\begin{figure}[htbp]
    \centering
    \resizebox{0.45\textwidth}{!}{%
        \input{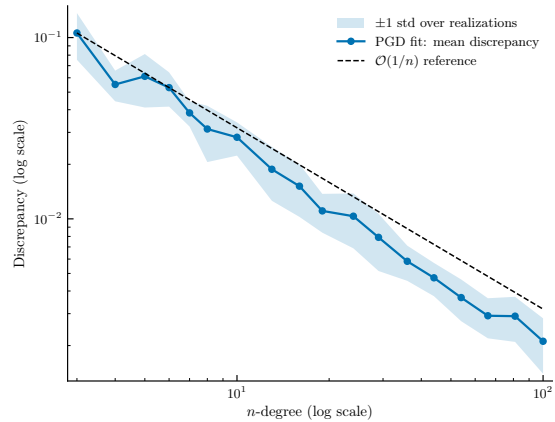}%
    }
    \caption{Restriction quality versus degree of the B\'ezier distribution.}
    \label{fig:restrictioneval}
\end{figure}
Figure~\ref{fig:restrictioneval} reports, as a function of $n$, the discrepancy
$D_n := \sup_{t\in[0,1]} \bigl| F(t;\mathbf{x},\mathbf{z}) - \hat{F}_n(t;\hat{\mathbf{x}},\hat{\mathbf{z}}) \bigr|$
between $F$ and its fit $\hat{F}_n$. We approximate $D_n$ on a uniform grid of $1{,}000$ points in $[0,1]$ and plot its mean over the $10$ replications with a standard deviation band. Unlike Theorem~\ref{theorem}, whose bound is joint in the $x$- and $z$-coordinates, $D_n$ captures only the vertical gap at each $t$, which we adopt for interpretability. The discrepancy decays at the predicted rate $\mathcal{O}(1/n)$ (black, for reference), confirming that additional control points compensate for the loss from restricting to monotone control points.

\section{Application to data-driven reliable shortest paths} \label{sec:spp}
To illustrate the practical utility of the \texttt{bezierv} package, we apply its fitting routines and the Monte Carlo convolution algorithm of Section~\ref{sec:conv} to a stochastic shortest path problem on a transportation network with random travel times. Leveraging the B\'ezier distributions fitted in Section~\ref{sec:transportationdata}, we construct a directed graph representation of the Chicago network in which each arc is associated with an independent B\'ezier random variable modeling its travel time. Formally, let $\mathcal{G} = (\mathcal{N}, \mathcal{A})$ denote the directed graph, where each arc $(i,j) \in \mathcal{A}$ is characterized by a deterministic cost $c_{ij} \in \mathbb{R}_{+}$ and a B\'ezier random variable $\tilde{t}_{ij}$ representing travel time. Given an origin-destination pair $(v_s, v_e)$, a time budget $T$, and a target reliability $\alpha \in [0,1]$, the shortest $\alpha$-reliable path (S-$\alpha$RP) \citep{Corredor2021} minimizes the total path cost $c(\mathcal{P}) = \sum_{(i,j) \in \mathcal{P}} c_{ij}$ while satisfying the chance constraint $
    \mathbb{P}[\tilde{t}(\mathcal{P}) \leq T] \geq \alpha,
$
where $\tilde{t}(\mathcal{P}) = \sum_{(i,j) \in \mathcal{P}} \tilde{t}_{ij}$. To solve this problem, we adopt the pulse algorithm framework \citep{Lozano2013}, a depth-first search approach with tailored pruning strategies to discard partial paths (pulses) that has been successfully applied to various stochastic shortest path problems \citep{Leiva2026}. 

To solve the S-$\alpha$RP, we implemented the pulse algorithm for which the \emph{infeasibility pruning} strategy evaluates each partial path by computing the probability of on-time arrival under a best-case travel-time scenario: any path whose optimistic probability still violates the reliability threshold is discarded outright. Therefore, for this pruning strategy, we need to compute the convolution of the travel times along a partial path, which we approximate with Monte Carlo using the convolution utility of \texttt{bezierv}.

\begin{figure}[htbp]
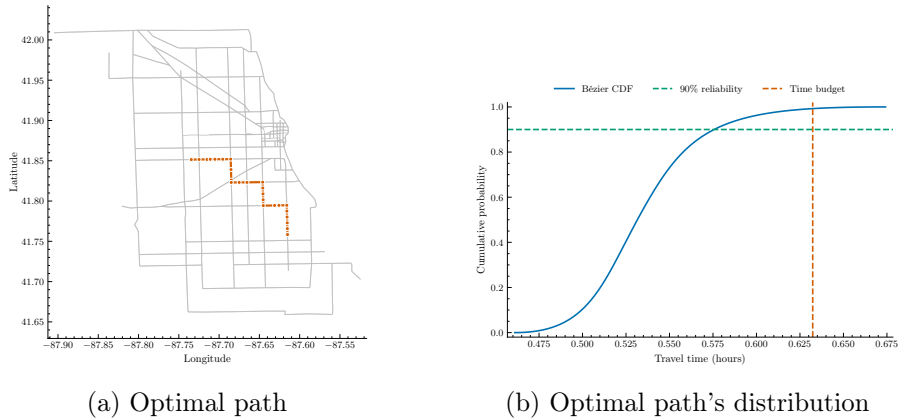

    \centering
    \begin{subfigure}[b]{0.3\textwidth}
        \resizebox{\textwidth}{!}{%
        \input{sarp_path_grid.pgf}}
        \caption{Optimal path}
        \label{fig:optpath}
    \end{subfigure}
    \hspace{1cm}
    \begin{subfigure}[b]{0.35\textwidth}
        \resizebox{\textwidth}{!}{%
        \input{sarp_path_dist.pgf}}
        \caption{Optimal path's distribution}
        \label{fig:optpath_dist}
    \end{subfigure}
    \caption{S-$\alpha$RP instance.}
    \label{fig:experimental_results}
\end{figure}

We illustrate the procedure for finding the shortest $\alpha$-reliable path on a trip from the city's Northwest Side (node 136) to the South Side (node 306) on the morning rush hour. We set the target reliability level at $\alpha = 0.9$. Following an established protocol \citep{Santos2007}, the time budget $T$ is calibrated as a weighted sum of the $\alpha$-quantiles from two benchmark paths (minimum mean-time and minimum cost), yielding $T = 0.632$ hours for this OD pair.  Figure~\ref{fig:experimental_results} shows the optimal path found using the pulse algorithm and the convolution utilities, which is 12.56 miles long, comprises 18 arcs, and achieves a reliability of 99.3\%, which satisfies the required threshold of 90\%.

\section{Concluding remarks} \label{sec:conclusions}
B\'ezier distributions are a flexible family for modeling bounded univariate data, but their adoption has been hindered by the lack of efficient fitting tools. We close this gap with a computational framework for both minimum error and maximum likelihood estimation via first-order methods. At the core of these methods is a monotonic restriction of the control points, which we prove to be an asymptotically lossless approximation of the B\'ezier distribution family. For minimum error estimation, enforcing monotonicity and fixing the $x$-coordinates of the control points yields a convex problem whose projection reduces to an isotonic regression, yielding an $\mathcal{O}(n)$ projection operator and enabling an efficient projected gradient descent (PGD) algorithm. For maximum likelihood estimation the same restriction applies, but the objective lacks a Lipschitz gradient, so standard PGD is unsuitable; we instead develop a primal gradient scheme based on relative smoothness that preserves the convergence guarantees. We also show that closed-form convolutions of B\'ezier variables are not necessarily tractable---they require roots of high-degree polynomials---and propose a Monte Carlo alternative, which proves essential in our data-driven reliable shortest path application.

Computationally, the proposed methods run three to four orders of magnitude faster than the original derivative-free approach of \cite{Wagner1993}. Against the nonlinear solver IPOPT, the first-order methods are substantially faster and more robust: although IPOPT attains the lowest objective value on the instances it solves, its runtime is orders of magnitude larger and it fails on a nontrivial fraction of instances, whereas PGD solves instances in roughly 3\,ms on real-world data and under 2\,ms on synthetic data without failures. For maximum likelihood estimation, the B\'ezier fit is comparable to the generalized beta and Johnson translation systems on unimodal data and substantially better on multimodal data at comparable fitting time. Moreover, our controlled evaluation confirms Theorem~\ref{theorem} empirically: the effect of the monotonicity restriction on fit quality is mitigated as the number of control points grows.

We package these advances in \texttt{bezierv}, an open-source Python library compatible with the scientific Python ecosystem (e.g., NumPy, SciPy) that bundles fitting, plotting, random-variate generation, and Monte Carlo convolution, lowering the barrier to entry for researchers and practitioners. We demonstrate its versatility both for standard input modeling and by embedding it within a larger algorithmic framework for stochastic shortest path problems, enabling the integration of B\'ezier distributions into broader decision-support systems.

Future research could extend this framework in several directions. Increasing the degree $n$ expands the family's expressive power but enlarges the parameter space, complicates the optimization landscape, and increases the risk of overfitting; an attractive alternative is to retain a moderate $n$ and jointly optimize the $x$- and $z$-coordinates rather than fixing the $x$-coordinates a priori to empirical quantiles. On the probabilistic front, copulas could capture the dependence structure between multiple B\'ezier random variables; on the computational front, the Monte Carlo convolution could be scaled for large-scale simulation experiments.

\bibliographystyle{informs2014}
\bibliography{refs}
\clearpage

\begin{APPENDIX}{}\label{apx:proofs}
\section{Examples of B\'ezier distributions} \label{appendix:examples_bezier}

    To provide intuition for the graphical nature of Bézier cdfs, we present three illustrative examples in Figure \ref{fig:bezierthree}, each displaying the cdf in black with its corresponding control points in red ($\mathbf{p}_i,\; i=0,\ldots,3$). The control points function as ``magnets'', pulling the curve towards their location. The pseudo-local control property is demonstrated by comparing the leftmost and center figures, where modifying the position of a single control   point (i.e., $\mathbf{p}_1$) alters the shape of the entire curve.
    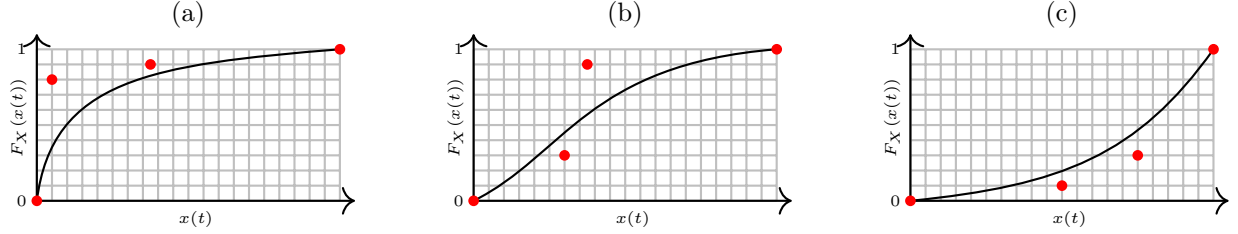
\begin{figure}[htbp]
          \centering
          \begin{subfigure}{0.3\textwidth}
            \caption{}
            \centering
            \resizebox{\linewidth}{!}{\begin{tikzpicture}
 \coordinate (P0) at (0,0);
 \coordinate (P1) at (0.10,0.8);
 \coordinate (P2) at (0.75,0.9);
 \coordinate (P3) at (2,1);

 \draw[gray!50, thin, step=0.1] (0,0) grid (2,1.0);
 \draw[->] (0,0) -- (2.1,0) node[scale=0.5,pos=0.5,below] {\tiny{$x(t)$}};
 \draw[->] (0,0) -- (0,1.1) node[scale=0.5,pos=0.5, sloped, above] {\tiny{$F_X(x(t))$}};

 \draw (P0) .. controls (P1) and (P2) .. (P3);

 \node[scale=0.5, below left] at (0,0.1) {\tiny{$0$}};
 \node[scale=0.5, left] at (0,1) {\tiny{$1$}};
 
 
 \fill[red] (P0) circle (1.0pt);
 \fill[red] (P1) circle (1.0pt);
 \fill[red] (P2) circle (1.0pt);
 \fill[red] (P3) circle (1.0pt);
\end{tikzpicture}}
          \end{subfigure}%
          \hfill
          \begin{subfigure}{0.3\textwidth}
            \caption{}
            \centering
            \resizebox{\linewidth}{!}{\begin{tikzpicture}
  \coordinate (P0) at (0,0);
  \coordinate (P1) at (0.6,0.3);
  \coordinate (P2) at (0.75,0.9);
  \coordinate (P3) at (2,1);

  \draw[gray!50, thin, step=0.1] (0,0) grid (2,1.0);
  \draw[->] (0,0) -- (2.1,0) node[scale=0.5,pos=0.5,below] {\tiny{$x(t)$}};
  \draw[->] (0,0) -- (0,1.1) node[scale=0.5,pos=0.5, above, sloped] {\tiny{$F_X(x(t))$}};
  \draw (P0) .. controls (P1) and (P2) .. (P3);

 \node[scale=0.5, below left] at (0,0.1) {\tiny{$0$}};
  \node[scale=0.5, left] at (0,1) {\tiny{$1$}};
  
  
  \fill[red] (P0) circle (1.0pt);
  \fill[red] (P1) circle (1.0pt);
  \fill[red] (P2) circle (1.0pt);
  \fill[red] (P3) circle (1.0pt);
\end{tikzpicture}}
          \end{subfigure}%
          \hfill
          \begin{subfigure}{0.3\textwidth}
            \caption{}
            \centering
             \resizebox{\linewidth}{!}{\begin{tikzpicture}
  \coordinate (P0) at (0,0);
  \coordinate (P1) at (1,0.1);
  \coordinate (P2) at (1.5,0.3);
  \coordinate (P3) at (2,1);

  \draw[gray!50, thin, step=0.1] (0,0) grid (2,1.0);
  \draw[->] (0,0) -- (2.1,0) node[scale=0.5,pos=0.5,below] {\tiny{$x(t)$}};
  \draw[->] (0,0) -- (0,1.1) node[scale=0.5,pos=0.5, above, sloped] {\tiny{$F_X(x(t))$}};
  \draw (P0) .. controls (P1) and (P2) .. (P3);

 \node[scale=0.5, below left] at (0,0.1) {\tiny{$0$}};
  \node[scale=0.5, left] at (0,1) {\tiny{$1$}};
  
  
  \fill[red] (P0) circle (1.0pt);
  \fill[red] (P1) circle (1.0pt);
  \fill[red] (P2) circle (1.0pt);
  \fill[red] (P3) circle (1.0pt);
\end{tikzpicture}}
          \end{subfigure}

          \caption{B\'ezier cdf examples.}
          \label{fig:bezierthree}
    \end{figure}

Figure \ref{fig:derivativethree} plots the pdf corresponding to each cdf from Figure \ref{fig:bezierthree}, visually confirming that each derivative is non-negative.

    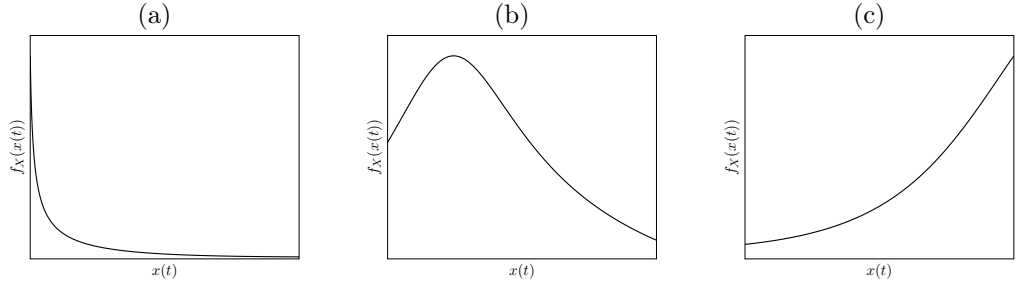
\begin{figure}[htbp]
        \centering
        \begin{subfigure}{0.24\textwidth}
            \caption{}
            \centering
            \resizebox{\linewidth}{!}{\begin{tikzpicture}
  \begin{axis}[
    xlabel={$x(t)$},
    ylabel={$f_X(x(t))$},
    grid=major,
    xmin=0,
    xmax=2,
    ymin=0,
    xticklabel=\empty,
    yticklabel=\empty,
    xtick=\empty,
    ytick=\empty,
    legend pos=north west,
    declare function={
      xoft(\t) = (1-\t)^3*0 + 3*\t*(1-\t)^2*0.1 + 3*\t^2*(1-\t)*0.75 + \t^3*2;
      numerator(\t) = 3*(1-\t)^2*(0.8-0) + 6*\t*(1-\t)*(0.9-0.8) + 3*\t^2*(1-0.9);
      denominator(\t) = 3*(1-\t)^2*(0.1-0) + 6*\t*(1-\t)*(0.75-0.1) + 3*\t^2*(2-0.75);
    },
  ]
    
  \addplot [
    black,
    thick,
    domain=0:1,
    samples=101,
    ]
    ({xoft(x)}, {numerator(x)/denominator(x)});
    
  \end{axis}
\end{tikzpicture}}
        \end{subfigure}%
        \quad\quad 
        \begin{subfigure}{0.24\textwidth}
            \caption{}
            \centering
            \resizebox{\linewidth}{!}{\begin{tikzpicture}
  \begin{axis}[
    xlabel={$x(t)$},
    ylabel={$f_X(x(t))$},
    grid=major,
    xmin=0,
    xmax=2,
    ymin=0,
        xticklabel=\empty,
    yticklabel=\empty,
    xtick=\empty,
    ytick=\empty,
    legend pos=north west,
    declare function={
      xoft(\t) = (1-\t)^3*0 + 3*\t*(1-\t)^2*0.6 + 3*\t^2*(1-\t)*0.75 + \t^3*2;
      numerator(\t) = 3*(1-\t)^2*(0.3-0) + 6*\t*(1-\t)*(0.9-0.3) + 3*\t^2*(1-0.9);
      denominator(\t) = 3*(1-\t)^2*(0.6-0) + 6*\t*(1-\t)*(0.75-0.6) + 3*\t^2*(2-0.75);
    },
  ]
    
  \addplot [
    black,
    thick,
    domain=0:1,
    samples=101,
  ]
  ({xoft(x)}, {numerator(x)/denominator(x)});

  \end{axis}
\end{tikzpicture}}
        \end{subfigure}%
        \quad\quad 
        \begin{subfigure}{0.24\textwidth}
            \caption{}
            \centering
            \resizebox{\linewidth}{!}{\begin{tikzpicture}
  \begin{axis}[
    xlabel={$x(t)$},
    ylabel={$f_X(x(t))$},
    xmin=0,
    xmax=2,
    ymin=0,
    xticklabel=\empty,
    yticklabel=\empty,
    grid=major,
    xtick=\empty,
    ytick=\empty,
    legend pos=north west,
    declare function={
      xoft(\t) = (1-\t)^3*0 + 3*\t*(1-\t)^2*1 + 3*\t^2*(1-\t)*1.5 + \t^3*2;
      numerator(\t) = 3*(1-\t)^2*(0.1-0) + 6*\t*(1-\t)*(0.3-0.1) + 3*\t^2*(1-0.3);
      denominator(\t) = 3*(1-\t)^2*(1-0) + 6*\t*(1-\t)*(1.5-1) + 3*\t^2*(2-1.5);
    },
  ]
    
  \addplot [
    black,
    thick,
    domain=0:1,
    samples=101,
  ]
  ({xoft(x)}, {numerator(x)/denominator(x)});
    
  \end{axis}
\end{tikzpicture}}
        \end{subfigure}

        \caption{B\'ezier pdf examples.}
        \label{fig:derivativethree}
    \end{figure}
\section{Proofs} \label{appendix:proofs}
\begin{proof}{Proof of Lemma~\ref{lem:CsubsetS}.}
The boundary conditions hold by definition of $\mathcal{C}_n$. Non-negativity of $P_z'$ on $[0,1]$ follows from $P_z'(t; n, \mathbf{z}) = n \sum_{i=0}^{n-1} B_{n-1,i}(t)\, \Delta z_i$. For $P_x'$, the boundary conditions give $\Delta x_j > 0$ for some $j\in \{0,...,n-1\}$; since $B_{n-1,i}(t) > 0$ on $(0,1)$, it follows that \sloppy$P_x'(t; n, \mathbf{x}) = n \sum_{i=0}^{n-1} B_{n-1,i}(t)\, \Delta x_i \ge n\, B_{n-1,j}(t)\, \Delta x_j > 0$ on $(0,1)$.
\Halmos
\end{proof}
\vspace{1cm}

\begin{proof}{Proof of Theorem~\ref{theorem}.}
    By the endpoint interpolation property, $({\mathbf{x}}, {\mathbf{z}})$ inherits the boundary conditions from $\mathcal{S}_{\tilde{n}}$. Since $P_x(\,\cdot\,;\, \tilde{n},\, \tilde{\mathbf{x}})$ is strictly increasing and $P_z(\,\cdot\,;\, \tilde{n},\, \tilde{\mathbf{z}})$ is non-decreasing on $[0,1]$, we have that for $i = 0, \ldots, n-1$, $\tfrac{i}{n} \leq \tfrac{i+1}{n}$ yields
    \begin{align*}
    {x}_i = P_x\!\bigl(\tfrac{i}{n};\, \tilde{n},\, \tilde{\mathbf{x}}\bigr) \;\leq\; P_x\!\bigl(\tfrac{i+1}{n};\, \tilde{n},\, \tilde{\mathbf{x}}\bigr) = {x}_{i+1},
    \qquad
    {z}_i = P_z\!\bigl(\tfrac{i}{n};\, \tilde{n},\, \tilde{\mathbf{z}}\bigr)\;\leq\; P_z\!\bigl(\tfrac{i+1}{n};\, \tilde{n},\, \tilde{\mathbf{z}}\bigr)= {z}_{i+1},
    \end{align*}
    which combined with the boundary values gives $({\mathbf{x}}, {\mathbf{z}}) \in \mathcal{C}_n$. To prove \eqref{eq:rate}, observe that by construction ${\mathbf{p}}_i = ({x}_i,\, {z}_i)^\top = \tilde{\mathbf{P}}(\tfrac{i}{n})$, so $
  {\mathbf{P}}(t) \;=\; \sum_{i=0}^{n} B_{n,i}(t)\, \tilde{\mathbf{P}}\!\left(\tfrac{i}{n}\right)$.
  Since each coordinate of $\tilde{\mathbf{P}}$ is a polynomial in $t$, $\tilde{\mathbf{P}} \in C^\infty([0,1];\, \mathbb{R}^2)$, and the extreme value theorem provides that
  \begin{align*}
  \kappa \;:=\; \max\bigl(\|P_x''(\cdot;\, \tilde{n},\, \tilde{\mathbf{x}})\|_\infty,\, \|P_z''(\cdot;\, \tilde{n},\, \tilde{\mathbf{z}})\|_\infty\bigr) \;<\; \infty.
  \end{align*}
  Fix $t \in [0,1]$ and apply Taylor's theorem with integral remainder to $P_x(\,\cdot\,;\, \tilde{n},\, \tilde{\mathbf{x}})$ at $t$:
  \begin{align}
  P_x\!\bigl(\tfrac{i}{n};\, \tilde{n},\, \tilde{\mathbf{x}}\bigr)
  \;=\;
  P_x(t;\, \tilde{n},\, \tilde{\mathbf{x}}) + P_x'(t;\, \tilde{n},\, \tilde{\mathbf{x}})\bigl(\tfrac{i}{n} - t\bigr) + R_i(t), \label{eq:taylor_x}
  \end{align}
  where the remainder is given by $R_i(t) \;:=\; \int_t^{i/n}\!\! P_x''(s;\, \tilde{n},\, \tilde{\mathbf{x}})\bigl(\tfrac{i}{n} - s\bigr)\, ds$ and  satisfies
  \begin{align}
      |R_i(t)| \leq \|P_x''(\cdot;\, \tilde{n},\, \tilde{\mathbf{x}})\|_\infty \tfrac{1}{2}(\tfrac{i}{n} - t)^2 \leq \tfrac{\kappa}{2}(\tfrac{i}{n} - t)^2. \label{eq:bound_second_der}
  \end{align}
  Multiplying \eqref{eq:taylor_x} by $B_{n,i}(t)$ and summing over $i=0,...,n$, and using the Bernstein moment identities
  \begin{align*}
  \sum_{i=0}^{n} B_{n,i}(t) \;=\; 1,
  \qquad
  \sum_{i=0}^{n} \!\bigl(\tfrac{i}{n} - t\bigr) B_{n,i}(t) \;=\; 0,
  \qquad
  \sum_{i=0}^{n} \!\bigl(\tfrac{i}{n} - t\bigr)^{\!2}\! B_{n,i}(t) \;=\; \frac{t(1-t)}{n},
  \end{align*}
  it follows that $P_x(t;\, n,\, {\mathbf{x}}) - P_x(t;\, \tilde{n},\, \tilde{\mathbf{x}} ) = \sum_{i=0}^{n} B_{n,i}(t)\, R_i(t)$. Moreover, using
   the bound \eqref{eq:bound_second_der} and the third Bernstein moment identity we obtain
  \begin{align*}
  \bigl|P_x(t;\, n,\, {\mathbf{x}}) - P_x(t;\, \tilde{n},\, \tilde{\mathbf{x}} ) \bigr|
  \;\leq\;
  \sum_{i=0}^{n} B_{n,i}(t)\, |R_i(t)|
  \;\leq\;
  \frac{\kappa}{2} \sum_{i=0}^{n} \!\bigl(\tfrac{i}{n} - t\bigr)^{\!2}\! B_{n,i}(t)
  \;=\;
  \frac{\kappa}{2}\frac{t(1-t)}{n}.
  \end{align*}
  An analogous argument applied to $P_z(\,\cdot\,;\, \tilde{n},\, \tilde{\mathbf{z}})$ yields $\bigl|
  P_z(t;\, n,\, {\mathbf{z}})- P_z(t;\, \tilde{n},\, \tilde{\mathbf{z}})\bigr| \leq \tfrac{\kappa}{2}\tfrac{t(1-t)}{n}$, so we obtain
  \begin{align}
      \bigl\|\mathbf{P}(t) - \tilde{\mathbf{P}}(t)\bigr\|_\infty \;\leq\; \frac{\kappa}{2}\frac{t(1-t)}{n}.\label{eq:final_bound}
  \end{align}
  The results then follows by taking the supremum over $t \in [0,1]$ in \eqref{eq:final_bound}.
\Halmos
\end{proof}
\vspace{1cm}

\begin{proof}{Proof of Corollary~\ref{corollary}.}
    Fix $j\in\{1,...,m\}$. By the triangle inequality and Theorem~\ref{theorem},
    \begin{align*}
        \Big| P_z(t_{(j)};n,\mathbf{z})-P_z(\tilde{t}_{(j)};\tilde{n},\tilde{\mathbf{z}}) \Big|
        &\leq \Big| P_z(t_{(j)};n,\mathbf{z})-P_z(t_{(j)};\tilde{n},\tilde{\mathbf{z}}) \Big|
        + \Big| P_z(t_{(j)};\tilde n,\tilde{\mathbf{z}})-P_z(\tilde{t}_{(j)};\tilde{n},\tilde{\mathbf{z}}) \Big| \\
        &\leq \frac{\kappa}{8n} + \Big| P_z(t_{(j)};\tilde n,\tilde{\mathbf{z}})-P_z(\tilde{t}_{(j)};\tilde{n},\tilde{\mathbf{z}}) \Big|.
    \end{align*}
    For the remaining term, the mean value theorem applied to $P_z(\cdot;\tilde n,\tilde{\mathbf{z}})$ and then to $P_x(\cdot;\tilde n,\tilde{\mathbf{x}})$ (using $P_x'\geq\tilde{\delta}$ and $P_x(\tilde{t}_{(j)};\tilde n,\tilde{\mathbf{x}})=X_{(j)}=P_x(t_{(j)};n,\mathbf{x})$), together with Theorem~\ref{theorem}, yields
    \begin{align*}
        \Big| P_z(t_{(j)};\tilde n,\tilde{\mathbf{z}})-P_z(\tilde{t}_{(j)};\tilde{n},\tilde{\mathbf{z}}) \Big|
        &\leq \Vert P_z'(\cdot;\tilde n,\tilde{\mathbf{z}})\Vert_\infty\,|t_{(j)}-\tilde{t}_{(j)}|\\
        &\leq \Vert P_z'(\cdot;\tilde n,\tilde{\mathbf{z}})\Vert_\infty \,\frac{\big| P_x(t_{(j)};\tilde n,\tilde{\mathbf{x}})-P_x(t_{(j)};n,\mathbf{x}) \big|}{\tilde{\delta}}\leq \frac{\rho\kappa}{8n}.
    \end{align*}
    Hence $\big| P_z(t_{(j)};n,\mathbf{z})-P_z(\tilde{t}_{(j)};\tilde{n},\tilde{\mathbf{z}}) \big|\leq \tfrac{(1+\rho)\kappa}{8n}$ for every $j=1,...,m$. Since $F_m$ and $P_z$ take values in $[0,1]$, factoring each squared difference gives
    \begin{align*}
        &\Bigg| \frac{1}{m}\sum_{j=1}^m \bigl( F_m(X_{(j)}) - P_z(t_{(j)};\,n,\,\mathbf{z}) \bigr)^2 - \frac{1}{m}\sum_{j=1}^m \bigl( F_m(X_{(j)}) - P_z(\tilde{t}_{(j)};\,\tilde{n},\,\tilde{\mathbf{z}}) \bigr)^2 \Bigg| \\
        &\quad\leq \frac{1}{m}\sum_{j=1}^m \Big| P_z(t_{(j)};n,\mathbf{z})-P_z(\tilde{t}_{(j)};\tilde{n},\tilde{\mathbf{z}}) \Big|\;\Big| 2F_m(X_{(j)})-P_z(t_{(j)};n,\mathbf{z})-P_z(\tilde{t}_{(j)};\tilde{n},\tilde{\mathbf{z}}) \Big|
        \leq \frac{(1+\rho)\kappa}{4n},
    \end{align*}
    where the last inequality uses $\big|2F_m(X_{(j)})-P_z(t_{(j)};n,\mathbf{z})-P_z(\tilde{t}_{(j)};\tilde{n},\tilde{\mathbf{z}})\big|\leq 2$.
\Halmos
\end{proof}
\vspace{1cm}

\begin{proof}{Proof of Proposition~\ref{prop:relsmooth}.}
    Let $\mathbf{w} \in \mathrm{int}(\Delta^{n-1})$ and consider $\nabla^2 l(\mathbf{w}) = \sum_{j=1}^m \frac{\mathbf{a}_j\mathbf{a}^\top_j}{(\mathbf{a}^\top_j \mathbf{w})^2}$ and $\nabla^2 h(\mathbf{w}) = \mathrm{diag}(\frac{1}{w_0^2},...,\frac{1}{w_{n-1}^2})$, which are well defined since $\mathbf{w}>\mathbf{0}$. For any $\mathbf{v} \in \mathbb{R}^n$ it holds that \begin{align*}
         \sum_{j=1}^m \left(\frac{\mathbf{a}_j^\top \mathbf{v}}{\mathbf{a}^\top_j \mathbf{w}}\right)^2 = \sum_{j=1}^m \left(\sum_{i=0}^{n-1}\frac{a_{ji} v_iw_i}{\mathbf{a}^\top_j \mathbf{w}w_i}\right)^2\leq \sum_{j=1}^m \sum_{i=0}^{n-1}\frac{a_{ji}w_i}{\mathbf{a}^\top_j \mathbf{w}}\left(\frac{ v_i}{w_i}\right)^2
        = \sum_{i=0}^{n-1} \frac{ v_i^2}{w_i^2} \left(\sum_{j=1}^m \frac{a_{ji}w_i}{\mathbf{a}^\top_j \mathbf{w}}\right) \leq m\sum_{i=0}^{n-1} \frac{ v_i^2}{w_i^2},
    \end{align*}
    where the first inequality follows from Jensen's inequality and the second inequality from $a_{ji}w_i \leq \mathbf{a}^\top_j \mathbf{w}$ for $j=1,...,m$. Thus, we conclude that
    $
        \mathbf{v}^\top\nabla^2 l(\mathbf{w})\mathbf{v} \leq m\sum_{i=0}^{n-1} \frac{ v_i^2}{w_i^2} = m\left[ \mathbf{v}^\top\nabla^2h(\mathbf{w})\mathbf{v}\right]$.
\Halmos
\end{proof}

\end{APPENDIX}

\end{document}